\documentclass[12pt]{amsart}
\usepackage{color}
\usepackage{graphicx}
\usepackage{amssymb}
\usepackage{amsthm}
\usepackage{amsmath}

\newtheorem{theorem}{Theorem}
\newtheorem{definition}{Definition}
\newtheorem{lemma}{Lemma}
\newtheorem{corollary}{Corollary}
\newtheorem{remark}{Remark}

\newcommand{\bea}{\begin{eqnarray*}}
\newcommand{\eea}{\end{eqnarray*}}
\newcommand{\ben}{\begin{eqnarray}}
\newcommand{\een}{\end{eqnarray}}
\newcommand{\beq}{\begin{equation}}
\newcommand{\eeq}{\end{equation}}

\newcommand{\R}{\ensuremath{\mathbb{R}}}

\newcommand{\supp}{\operatorname{supp}}

\newcommand{\Id}{\operatorname{Id}}

\newcommand{\cB}{\mathcal{B}}

\newcommand{\cF}{\mathcal{F}}
\newcommand{\cG}{\mathcal{G}}
\newcommand{\cK}{\mathcal{K}}

\newcommand{\cM}{\mathcal{M}}

\newcommand{\cX}{\mathcal{X}}

\newcommand{\fm}{\mathfrak{m}}
\newcommand{\fn}{\mathfrak{n}}

\newcommand{\half}{\frac{1}{2}}
\newcommand{\D}{\R^2_+}
\newcommand{\mn}{|\!|\!|}

\newcommand{\bu}{\mathbf{u}}
\newcommand{\by}{\mathbf{y}}
\newcommand{\bx}{\mathbf{x}}
\newcommand{\bX}{\mathbf{X}}
\newcommand{\bz}{\mathbf{z}}

\newcommand{\til}[1]{\widetilde{#1}}
\renewcommand{\hat}[1]{\widehat{#1}}

\newcommand{\al}{\alpha}

\newcommand{\de}{\delta}
\newcommand{\eps}{\varepsilon}
\newcommand{\si}{\sigma}

\newcommand{\Up}{\Upsilon}

\newcommand{\om}{\omega}

\renewcommand{\d}{\partial}

\begin{document}

\title{Blowup with vorticity control for a 2D model of the Boussinesq equations}
\author{V. Hoang}
\email{vu.hoang@rice.edu}
\address{Rice University
Department of Mathematics -- MS 136
P.O. Box 1892
Houston, TX 77005-1892}

\author{B. Orcan-Ekmekci}
\email{orcan@rice.edu}
\address{Rice University
Department of Mathematics -- MS 136
P.O. Box 1892
Houston, TX 77005-1892}

\author{M. Radosz}
\email{maria\_radosz@hotmail.com}
\address{Rice University
Department of Mathematics -- MS 136
P.O. Box 1892
Houston, TX 77005-1892}

\author{H. Yang}
\email{hy18@rice.edu}
\address{Rice University
Department of Mathematics -- MS 136
P.O. Box 1892
Houston, TX 77005-1892}

\date{\today}

\begin{abstract}
We propose a system of equations with nonlocal flux in two space dimensions which is closely modeled after the 2D Boussinesq equations in a hyperbolic flow scenario. Our equations involve a simplified vorticity stretching term and Biot-Savart law and provide insight into the underlying intrinsic mechanisms of singularity formation. We prove stable, controlled finite time blowup involving upper and lower bounds on the vorticity up to the time of blowup for a wide class of initial data.
\end{abstract} 
\maketitle
\section{Introduction}
The two-dimensional Boussinesq equations for the vorticity $\om$ and density $\rho$
\beq\label{Bouss}
\begin{split}
\om_t + \bu\cdot \nabla \om &= \rho_{x_1}\\ 
\rho_t +\bu\cdot \nabla \rho &= 0
 \end{split}
\eeq
with velocity field $\bu=(u_1,u_2)$ given by
\beq\label{VelBouss}
\bu = \nabla^{\bot} (-\Delta)^{-1} \om
\eeq
are an important system of partial differential equations arising in atmospheric sciences and geophysics, where the
system models an incompressible fluid of varying, temperature dependent density subject to gravity. 
On the other hand, the system \eqref{Bouss} exhibits some mathematical features of great interest. It is well-known that the three-dimensional axisymmetric Euler equations in a cylinder are almost identical to \eqref{Bouss} at least at points away from the cylinder axis \cite{MajdaBertozzi}. That is,
\eqref{Bouss} contains an analogue of the 3D Euler vorticity stretching mechanism in the right-hand side term $\rho_{x_1}$ of the first equation of \eqref{Bouss}.

The question whether solutions of the 3D Euler equations blow up in finite time from smooth data with finite energy is a profound, as of yet unanswered question \cite{ConstRev}. Recent progress was made by
by T. Hou and G. Luo  using numerical simulations. In \cite{HouLuo1} the authors compute approximate solutions for 
which the magnitude of the vorticity vector appears to become infinite close to an intersection of the domain 
boundary with a symmetry plane at $z=0$. Their blowup scenario is referred to as the \emph{hyperbolic flow scenario}. We refer to 
Figure \ref{figure_cylinder} for an illustration of the flow geometry.
\begin{figure}[htbp]
\includegraphics[scale=0.5]{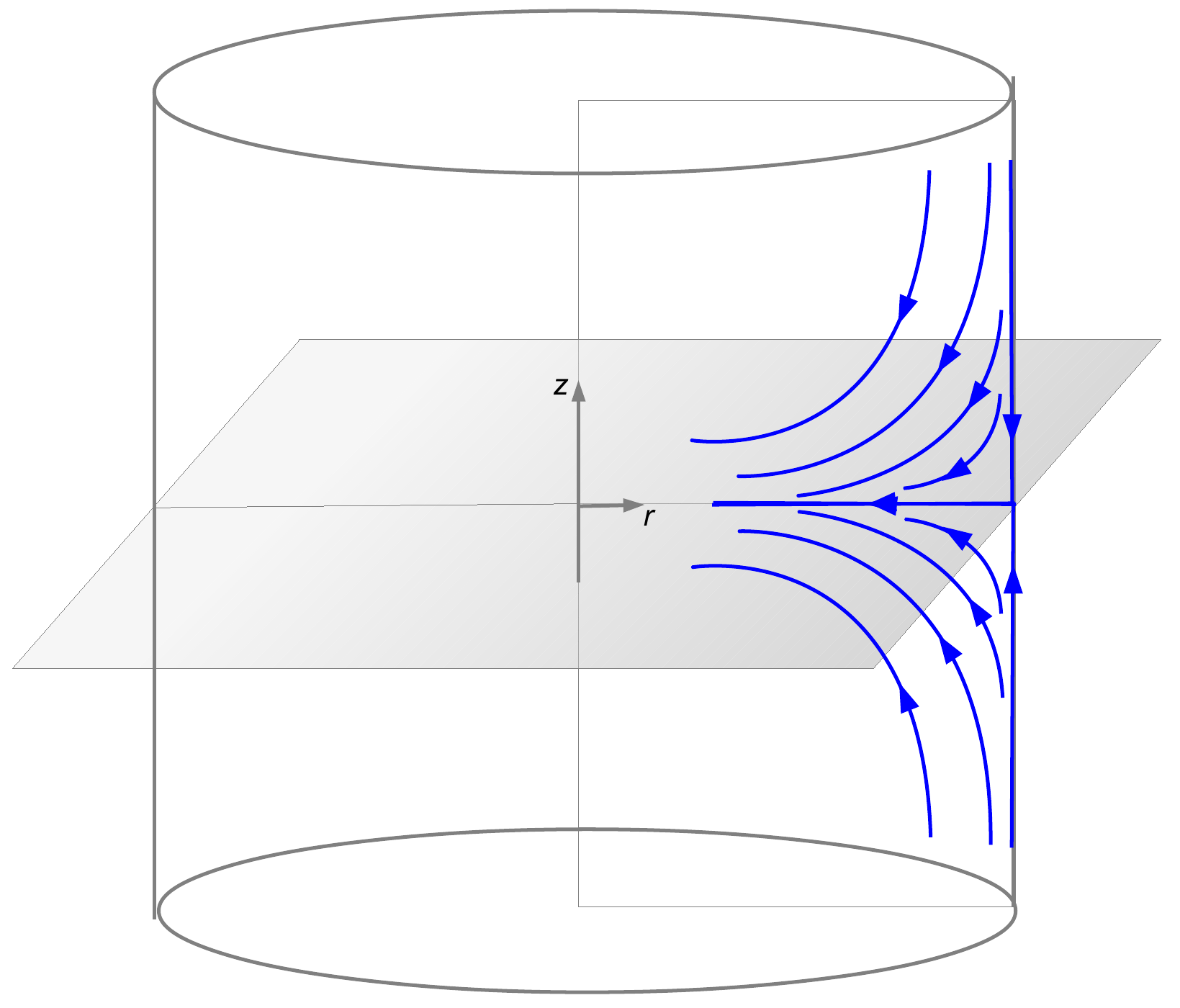}
\caption{Hyperbolic flow for axisymmetric 3D Euler.}
\label{figure_cylinder}
\end{figure}

A promising idea that leads to a better understanding of the singularity formation was introduced by A. Kiselev and V. \v{S}ver\'{a}k
in \cite{kiselev2013small}. There, the hyperbolic flow scenario for the 2D Euler equation
\beq\label{2DEuler}
\om_t + \mathbf{u}\cdot \nabla \om = 0
\eeq
with velocity field \eqref{VelBouss} was considered. More precisely, the authors studied the flow
close to the intersection of a symmetry axis of the solution with a domain boundary. The setup
created a stagnant point of the flow, towards which the flow is directed. 

For the two-dimensional Boussinesq equation, there is an additional vorticity stretching term on the right-hand side of the vorticity transport equation.  It is possible that methods based on the construction in \cite{kiselev2013small}, exploiting the hyperbolic flow scenario, can be used to prove finite time blowup for the
two-dimensional Boussinesq equations. Owing to the possible
growth in vorticity and various other issues, it seems that
there are many obstacles along the path.

It is reasonable to approach the problem by studying first
a simplified system of equations, which captures several
essential aspects of the hyperbolic flow scenario, while 
keeping the nonlocal and nonlinear features of the original
equations. The model equations we consider here are
inspired by the two-dimensional Boussinesq equations but contain a vorticity streching term of a different form.

The system reads as follows:
\beq\label{model}
\begin{split}
\om_t + \mathbf{u}\cdot \nabla \om &=\frac{\rho}{x_1}\\ 
\rho_t +\mathbf{u}\cdot \nabla \rho &= 0
\end{split}
\eeq
Note that the vorticity stretching term  on the right-hand side of the
equation for $\om$ arises by replacing the slope $\rho_{x_1}$ by 
$\frac{\rho}{x_1}$.

The functions $\om, \rho$ are defined on $[0, \infty)^2 = \R_+^2$,
and we consider solutions whose support initially is bounded away from the vertical axis $x_1=0$.
We will also consider a simplified velocity field
of the form $\bu = (u_1, u_2)$,
\begin{align}\label{def_u}
\begin{split}
-u_1(\bx) &= x_1Q(\bx, t)\\
u_2(\bx) &= x_2Q(\bx, t),
\end{split}
\end{align}
where $Q$ is defined by
\begin{align}\label{def_Q}
Q(\bx,t)=\iint_{\bx+S_\alpha} \frac{y_1y_2}{|\by|^4} \om(\by,t)~d\by
\end{align}
and $S_\alpha$ is the sector
\begin{align}\label{sector}
S_\alpha = \{ (x_1, x_2) : x_1\geq 0, ~~0 \leq x_2 \leq \alpha x_1\}
\end{align}
with arbitrary large, fixed $\alpha > 0$.

\subsection{Motivation and discussion}
In this part, we motivate the introduction of the system consisting of
\eqref{model}, \eqref{def_u} and \eqref{def_Q}. 
First we would like to remark that the simplified velocity
field \eqref{def_u}, \eqref{def_Q} is motivated by \cite{kiselev2013small}. There the authors consider the hyperbolic flow scenario for the two-dimensional
Euler equation in vorticity form on a disc. A stagnant point of the flow is created at the intersection point of a symmetry axis and the boundary
of the disc. The proof in \cite{kiselev2013small} is
accomplished by tracking the evolution of a ``projectile", a region on which $\om = 1$. The projectile is carried towards the stagnant point, and due to the nonlinear interaction, the projectile is fast enough to create double-exponential growth of $\nabla \om$.

As mentioned before, it is conceivable that the hyperbolic scenario might be used to prove blowup for
the Boussinesq system. By choosing an initial profile
for $\rho$ that is monotone increasing for $x_1>0$, 
one would hope that $\rho_{x_1}$ remains mostly
positive up to the time of blowup. The compression of
the hyperbolic flow should lead to a growth in $\rho_{x_1}$ which translates into a growth of $\om$
via the vorticity stretching. This in turn should lead to further, amplified growth in $\rho_{x_1}$. 
In order to prove finite-time blowup however, one has
to carefully quantify the growth rates of these quantities. The problem is compounded by the nonlocal
nature of the Biot-Savart law. 
\begin{figure}[htbp]
\includegraphics[scale=0.8]{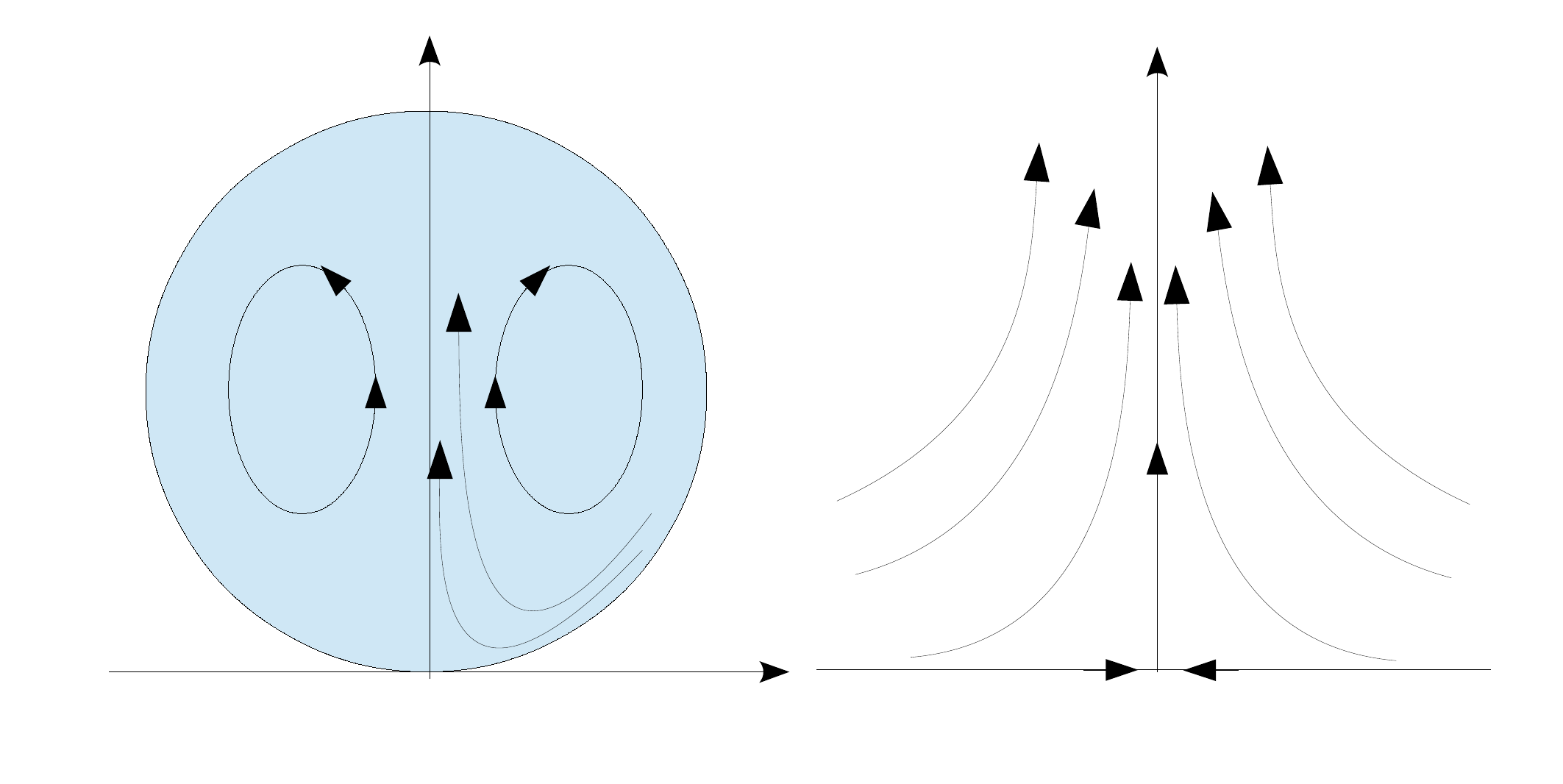}
\caption{Hyperbolic flow on a disc. The picture
on the right is a closeup of the flow around
the stagnant point.
\label{figure_hypDisc}}
\end{figure}
For this reason we have chosen to study a simplified
system inspired by the Boussinesq system, which retains a similar nonlocal and nonlinear structure. 
In the following, we would like to further remark on
these issues:
\begin{itemize}
\item
A central result of \cite{kiselev2013small} is the
representation
\begin{align}\label{rep_KiselevSverak}
\begin{split}
u_1(\bx, t) = - x_1\iint_{y_1\geq x_1, y_2\geq x_2}\frac{y_1 y_2}{|\mathbf{y}|^4}\om(\mathbf{y}, t) ~d\by + x_1 B_1(\bx, t)\\
u_2(\bx, t) = x_2\iint_{y_1\geq x_1, y_2\geq x_2}\frac{y_1 y_2}{|\mathbf{y}|^4}\om(\mathbf{y}, t) ~d\by + x_2 B_1(\bx, t)
\end{split}
\end{align}
for the velocity field. Here, the integrals are the main contribution to the 
velocity field in the sense that they keep the 
hyperbolic flow scenario going, and
$B_1$ and $B_2$ are error terms which can be bounded by $C\|\om\|_{L^\infty}$. 

Thus, for the two-dimensional Euler equations the error terms can be controlled owing to the conservation of the $L^\infty$-norm of the vorticity. For the Boussinesq system, this is not possible due to the presence of the vorticity stretching term. This means that for the original Biot-Savart law it is not straightforward to guarantee the continued persistence of the hyperbolic flow scenario. During a phase of 
strong vorticity growth, the error contributions might become
dominant.

Our velocity field \eqref{def_u} is motivated by the structure of the main terms in \eqref{rep_KiselevSverak}, although not identical. By writing the velocity field in the form \eqref{def_u},
we create a stagnant point of the flow at $\bx = (0, 0)$, and for nonnegative $\om$ we obtain the 
compression and expansion along the two coordinate
directions typical for the hyperbolic flow scenario.
We note that the role of the symmetry axis in our setup is played by the vertical coordinate axis, whereas
the horizontal axis corresponds to the boundary of the disc from \cite{kiselev2013small}.

In summary, our velocity field \eqref{def_u} allows
us to set up the hyperbolic scenario in a natural fashion. 
Control of vorticity growth will play an important role in our proof, and we believe that the techniques we develop here will allow us at a later stage to successfully control the errors, and to revert back to the original Biot-Savart-Law.
\item Let us now remark on the simplified vorticity streching mechanism in \eqref{model}.
The most obvious advantage is that the sign of $\frac{\rho}{x_1}$ is definite, if $\rho$ has a definite sign. This is not true for the original vorticity streching $\rho_{x_1}$, and at this stage the control of the sign of $\rho_{x_1}$ proves to be challenging. The simplified vorticity streching creates a more direct connection between the compression of the fluid and the feedback into vorticity growth. However, we do not expect that the fine structure of the singularity of the model problem \eqref{model} corresponds to that 
of the Boussinesq problem. We refer to the forthcoming paper
\cite{HoangRadosz1D} for a more thorough discussion of this aspect of the problem.
\end{itemize}

\subsection{Simplified picture of blowup mechanics.} Our technique emphasizes the idea to control the vorticity over \emph{regions of space and up to
the time of blowup}. The importance of this idea is readily illustrated by the following  considerations (we follow the presentation in 
\cite{ConstRev} by P.~Constantin). Suppose we wish to study vorticity growth in the context of the 3D Euler
equation
\begin{align}\label{3DEuler}
\om_t + (\bu\cdot\nabla) \om = (\om \cdot \nabla) \bu. 
\end{align}
For the magnitude of vorticity, we have the equation 
\begin{align}
(\partial_t + \bu \cdot \nabla) |\om| = \al(x, t)|\om|
\end{align}
where the vorticity stretching factor $\al$ is given by a singular integral operator acting on the vorticity field $\om$. The heuristic idea to achieve blowup therefore would be to look for situations where 
\begin{align}\label{heur}
\frac{d}{dt} |\om| \geq c|\om|^2.
\end{align} 
However, since $\al$ depends in a nonlocal way on $\om$, it is not clear for which vorticity distributions we should expect \eqref{heur} to hold for a long enough time
to produce infinite growth of $|\om|$. Indeed, geometric properties of the velocity field can lead to absence of blowup \cite{ConstantinFefferman, CFM}.

Our overall strategy carried out in the context of our simplified Boussinesq system will therefore be to study particular flow scenarios, such as the hyperbolic flow scenario at a domain boundary, and to find ways to control the distribution of $\om$ in space. The geometric properties of the velocity field in this situation are not in contradiction to a blowup (see \cite{HouLuo1,HouLuo2}). The control of $\om$ over some significant areas in space will create leverage to achieve finite time blowup.
 
Let us now give a simplified picture of the blowup for \eqref{model}, where the control is achieved via barrier functions.
To this end, we consider only the behavior of $\om$ on the horizontal axis, i.e. 
$\om(x_1, 0, t)$. The barrier functions that we construct have the follwing form for $x_2 = 0$:
\begin{align}\label{c1}
\phi_0 x_1^{-p_0} < \om(x_1, 0, t) < \phi_1 x_1^{-p_1}
\end{align}
with positive powers $p_0 < \frac{1}{2} < p_1$ satisfying $p_0 + p_1 = 1$. 
\eqref{c1} holds for 
\begin{align}\label{eq:barrierdomain}
b(t) \leq x_1 \leq 1,
\end{align} where $b(t)>0$ is a strictly decreasing function  which will be defined below (for an illustration, see Figure \ref{f3}). Slightly different control conditions hold for $x_1 > 1$ due to technical reasons.

\begin{figure}[htbp]
\includegraphics[scale=0.4]{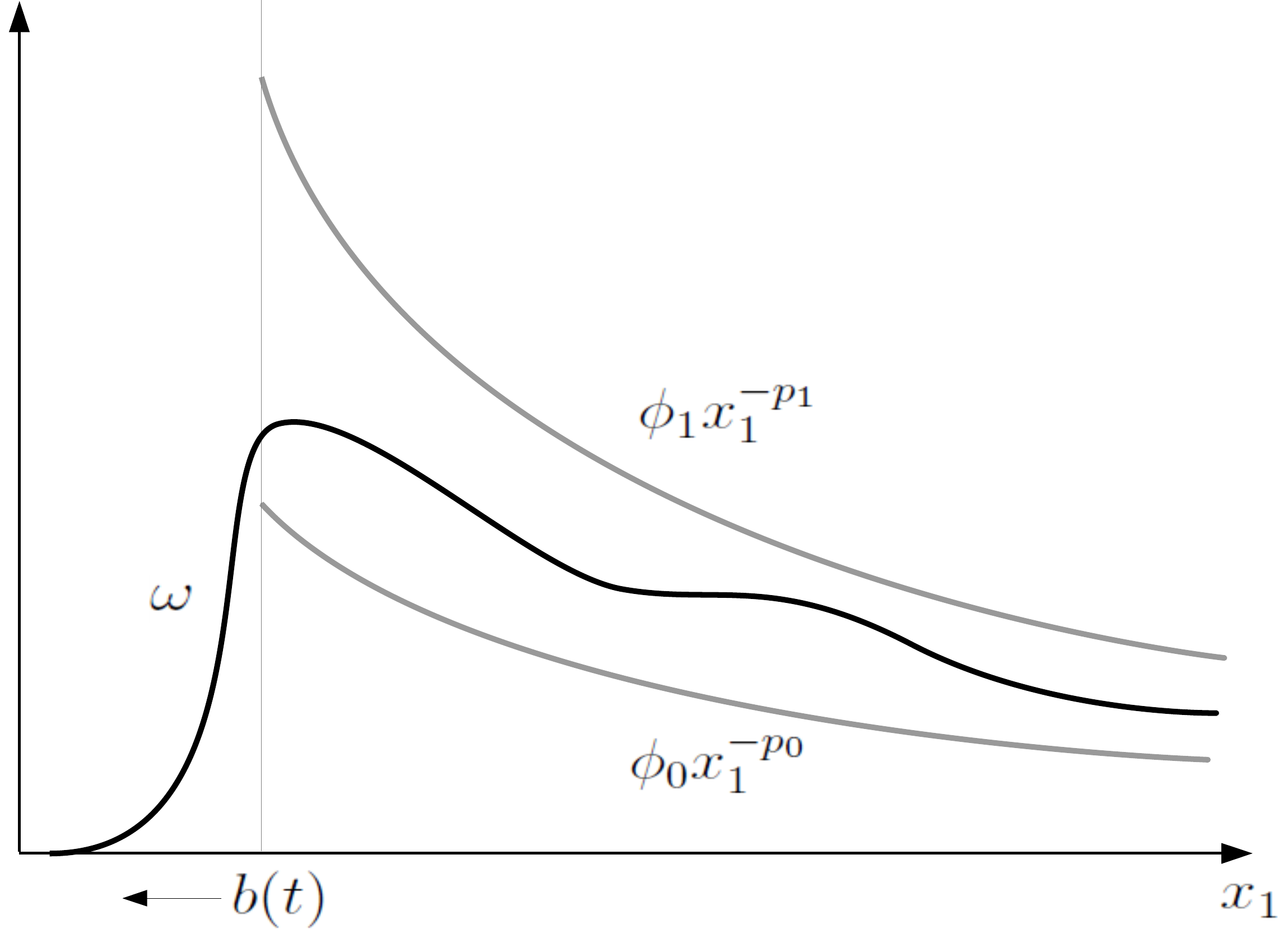}
\caption{$\om$ is enclosed between $\phi_0x_1^{-p_0}$ and $\phi_1x_1^{-p_1}$ for $x_1\ge b(t)$.
\label{f3}}
\end{figure}

The key to the proof will be to show that if $\om$ was initially enclosed between the lower and the upper barrier, it will remain so for as long as it remains smooth. Note that we start from initial data with support contained in $\{x_1 > 0\}$. 

While the barrier functions are constant in time, the barrier domain \eqref{eq:barrierdomain} is dynamic, and one of our main tasks is to find a suitable evolution equation for $b(t)$. To achieve this we take the behavior of $\om(x_1, x_2, t)$ for $x_2 > 0$ into account, where further suitable control conditions are in place.      
It turns out that $b(t)$ will satisfy 
$$b(t) \sim -b(t)^{1-p_0}$$ which is caused by the compression of hyperbolic flow. This implies that $b(t)$ reaches zero in finite time. Informally, we can say that the lower bound of the barrier pushes the values of $\om$ to $+\infty$. While only the lower bound is needed to show blowup, the upper bound is crucial to stabilize the scenario, ensure continued control over $\om$ and to derive the lower bound. Of course, the smooth solution can break down before $b(t)$ becomes zero. In this case, we use a criterion of Beale-Kato-Majda type (see Theorem \ref{thm_localExistence}) to show that the vorticity becomes infinite at the blowup time.

\subsection{Related results.}
The construction in this paper is inspired a by a corresponding result by V. Hoang and M. Radosz for a one-dimensional system \cite{HoangRadosz1D}. The one-dimensional case
was in turn strongly inspired by the work of S.~Denisov
\cite{Denisov1,Denisov2,Denisov3},
especially his idea of starting from a singular steady solution. For our system \eqref{model}, no explicit singular steady state solution is known, but for the one-dimensional analogue
\begin{align*}
\begin{split}
\om_t + u \om_{x} &=\frac{\rho}{x}\\ 
\rho_t +u \rho_{x} &= 0
\end{split}
\end{align*}
with Biot-savart-Law
\begin{align*}
u(x, t) = - x\int_{x}^\infty \frac{\om(y, t)}{y}~dy,
\end{align*}
a steady singular solution is 
\begin{align}\label{sing_solution}
\om(x) = x^{-\frac{1}{2}}, \rho(x) = c,
\end{align}
$c>0$ being a suitable constant. A very natural question arises:
what happens if we smooth out the steady solution profile $\om$ to obtain a $\om_0$ which is compactly inside $(0, \infty)$ and use $\om_0$ as initial data for the evolution?
If the singular steady state is ``stable'' in a suitable 
sense, one might conjecture that the smooth solution approaches the singular steady solution, causing blowup in finite time.

The stationary solution also motivates the form of the barrier functions used in the blowup proof.  It is
crucial that $p_0 < \frac{1}{2} < p_1$, indicating that the power $\frac{1}{2}$  in \eqref{sing_solution} plays a 
special role.

The investigation of equations with simplified Biot-Savart laws of the type considered here was begun in \cite{CKY}, where a simplified version of a one-dimensional model given in \cite{HouLuo2} was investigated (see also \cite{sixAuthors}).

There seem to be few results that deal with finite time
blowup in the context of two-dimensional active scalar equations or systems of equations. 
D. Chae, P. Constantin and J. Wu \cite{ChaeConstWu} present an example of a two-dimensional active scalar equation with nonlocal flux with finite-time blowup.
In \cite{KiselevLenyaYao}, A.~Kiselev, L.~Ryzhik, Y.~Yao and A.~Zlato\v s prove finite time blowup for patch solutions of the modified SQG equations.
 
Independently from our work, A. Kiselev and C. Tan \cite{KiselevTan} prove finite time
blowup for the system \eqref{model} using another
simplified Biot-Savart law. Their approach is completely different from ours (barrier functions) and their model has different properties. For instance, their velocity field in incompressible, whereas ours is not. On the other hand, the ``2D Euler equation'' derived by setting $\rho\equiv 0$ and using their Biot-Savart law exhibits only exponential gradient growth. The version of the 2D Euler equation with our Biot-Savart law, has solutions with double exponential gradient growth. This can be shown similar to the results in \cite{kiselev2013small}.
  
The recent preprint \cite{TT} explores a class of Lagrangian models for the 3D Euler equations, which are not directly related to our hyperbolic flow scenario. 
\subsection{Plan of the paper.}In Section \ref{sec_problem} we state the problem and main results. In Section \ref{sec_local} 
we address the local existence of smooth solutions 
of our evolution problem. Finally, in Section \ref{sec_sing}, we construct singular solutions. 

\section{Problem statement and main results}\label{sec_problem}
We consider classical solutions 
\begin{align}\label{eq:soln}
\om\in C([0, T], \cX), \rho \in C([0, T], \cX) 
\end{align}
where $\cX$ denotes the class of functions $f$ such that
\begin{itemize}
\item
$f=\til f|_{\D}$ where $\til f\in C^1(\R^2)$ has compact support such that
$\supp \til f \subseteq \{x_1>0\}$.
\end{itemize}
Observe that functions in $\cX$ in general do not vanish on the horizontal axis. 
Throughout the paper, we write $\|\cdot\|_\infty = \|\cdot\|_{L^\infty(\D)}$ and $\mn \cdot \mn$ for any matrix norm on real $2\times 2$-matrices. For matrix-valued 
functions $F(\bx)$, we shall also write
$$
\mn F \mn_{L^\infty(\D)} := \sup_{\bx\in \D}\mn F(\bx) \mn. 
$$
The following local existence and Beale-Kato-Majda continuation result holds:
\begin{theorem}\label{thm_localExistence}
Given $(\om_0, \rho_0)\in \cX\times \cX$, there exists a $T = T(\om_0, \rho_0,\al)>0$ and a unique solution $(\om, \rho)$ of \eqref{model} satisfying \eqref{eq:soln} with initial data $\om_0, \rho_0$. Moreover, for solutions with nonnegative 
initial data $\om_0(\bx) \geq 0, \rho_0(\bx) \geq 0$, the following
continuation criterion holds: The existence of a $k>0$ such that
\begin{align}\label{cond_continuation}
\int_0^T \|\om(\cdot, s)\|_{L^\infty([0, k]\times [0, \infty))} ~ds < \infty
\end{align}
implies that the solution can be continued with the same spatial smoothness \eqref{eq:soln} to a slightly
larger time interval $[0, T+\eps)$, $\eps>0$.
\end{theorem}

Our main result states the existence of finite-time blowup and gives a description of a class of initial data that blows up in finite time. 
\begin{definition}\label{def:D0}
Let $b_0, p_1, p_0, \de, \phi_0, \phi_1$ be positive parameters satisfying
\begin{align}\label{ConstantConditions}
p_0+p_1 = 1, ~~p_0\in (0, 1/2), ~~  0 < \de < 1, 0 < b_0 < \de, 0<\phi_0 < \phi_1.
\end{align}
Let $D_0$ be the region
\begin{align*}
\bx \in D_0 \Leftrightarrow \left\{\begin{array}{ll}
0 \le x_2 \le 2\alpha(x_1-b_0) & b_0 < x_1 \leq \de\\
0\le x_2 \le 2\alpha x_1 & \de < x_1 < 4 
\end{array}\right..
\end{align*}
where $\al > 0$ is the parameter in \eqref{sector}.

We call initial data $\om_0, \rho_0\in\cX$ \emph{suitably prepared}, if $\om_0, \rho_0$ are non-negative and 
the following conditions are satisfied:
\begin{align}\label{cond_initial_data}
\begin{split}
\begin{array}{lllr}
\phi_0 x_1^{-p_0} &<& \om_0(\bx) &\quad \bx\in D_0\cap \{x_1 < 1\},\\
\phi_0 &<& \om_0(\bx) &\quad \bx \in D_0 \cap \{1\leq x_1 < 3\},\\
\om_0(\bx) &<& \phi_1 x_1^{-p_1}&\quad \bx\in D_0\cap \{x_1 < 1\},\\
\om_0(\bx) &<& \phi_1 &\quad \bx \in D_0\cap \{1\leq x_1<4\},\\
\rho_0(\bx) &=& 1 &\quad \bx\in D_0.
\end{array}
\end{split}
\end{align}
In addition, assume $\supp \om_0,~\supp \rho_0 \subset (0, 4)\times [0, \infty)$ and
$\rho_0(x)\le 1$ everywhere. 
\end{definition}
We refer to figure \ref{f2} for an illustration of $D_0$. 
Our main blowup result reads:
\begin{theorem}\label{theorem_Blowup}
For all $\alpha > 0$, we have the following statement: There exist positive parameters $b_0, p_1, p_0, \de, \phi_0, \phi_1$ with conditions  $\eqref{ConstantConditions}$
such that for any suitably prepared initial data $\om_0, \rho_0\in \cX$
the corresponding smooth solution $(\om, \rho)$ blows up in finite time, i.e.
\begin{align}\label{eq:vorticityBlowup}
\lim_{t\to T_s} \|\om(\cdot, t)\|_{\infty} = \infty,
\end{align}
where $0 < T_s <\infty$ is the lifespan of the solution.

Moreover, there exists $T^*>0$ and a smooth, positive, strictly decreasing function $b:[0, T^*)\to (0, \infty)$ with $\lim_{t\to T^*}b(t) = 0$  and $T_s \le T^*$ such that
\begin{align}\label{eq:theorem_Blowup}
\phi_0 x_1^{-p_0} < \om(x_1, 0, t) < \phi_1 x_1^{-p_1}  \quad (b(t) < x_1 < 1) 
\end{align}
holds for $t\in [0, T_s)$. 
\end{theorem}

\begin{remark}
\begin{enumerate}
\item The essence of the conditions \eqref{cond_initial_data} is that the initial vorticity is enclosed between the two singular power functions $\phi_0 x_1^{-p_0}$ and $\phi_1 x_1^{-p_1}$ for $b_0\le x_1$, where $b_0 > 0$ is small. For $x_1 < b_0$, $\om_0$ goes smoothly to zero. From the proof of Theorem \ref{theorem_Blowup} it is seen that the gap between upper and lower bound can be arbitrarily wide as long as $b_0$ is chosen sufficiently small. This allows for a wide class of initial profiles.
\item Theorem \ref{theorem_Blowup} shows that the blowup is stable 
in $\om_0$, i.e.~small changes in the initial vorticity still lead to blowup. We could have easily also allowed variations in the values of $\rho_0$, e.g.~we could have required $1-\eps < \rho_0(\bx) < 1+\eps$ for $\bx\in D_0$. 
\item
$T^*$ can be computed in terms of known quantities, see
\eqref{def_Tstar}. We remark that $T^*$ is an upper bound on the lifespan of the solution, and not the 
precise time of regularity breakdown, i.e.~$b(T_s)$ may not be zero. If $T_s=T^*$, then \eqref{eq:theorem_Blowup} holds 
for $x_1 \in (0, 1)$, and the lower bound in Theorem \ref{theorem_Blowup} would imply $\|\om(\cdot, T_s)\| = \infty$. We know \eqref{eq:vorticityBlowup}, however, independently of
\eqref{eq:theorem_Blowup} (see the proof of Theorem
\ref{theorem_Blowup}).
\item The slope $2\al$ in the definition of $D_0$ is for ease of presentation only. We could have chosen $k_0\al$ for any $k_0>1$. The same applies to the definition of $g$ and $D_t$ in Section \ref{sec_sing}.
\end{enumerate}
\end{remark}
\section{Local Existence of Solutions}
\label{sec_local}
\subsection{Particle trajectory method.}
Following the \emph{particle trajectory method} (see also \cite{MajdaBertozzi}), we first derive equations for 
the flow map $\Phi=(\Phi_1,\Phi_2)$ 
\begin{align*}
\frac{d\Phi}{dt}(\bz,t)=u(\Phi(\bz,t),t), \; \Phi(\bz,0)=\bz,
\end{align*}
or equivalently
\begin{align}\label{eq:Phi1}
\Phi(\bz,t)=\bz+ \int_0^t \bu(\Phi(\bz,s),s)~ds.
\end{align}
The velocity field is given by
\beq \label{eq:u}
\bu(\bx,t):= \left(-x_1 \int_{\bx+S_\al} \frac{y_1y_2}{|\by|^4}\om(\by,t)~d\by, x_2 \int_{\bx+S_\al} \frac{y_1y_2}{|\by|^4}\om(\by,t)~d\by\right).
\eeq
By integrating the first equation of \eqref{model}
along a trajectory and observing that $\rho$ is transported, we get
\beq\label{eq:om2}
\om(\by,t):=\om_0(\Phi^{-1}(\by,t))+\int_0^t\frac{\rho_0(\Phi^{-1}(\by,t))}{\Phi_1(\Phi^{-1}(\by,t),s)}~ds.
\eeq 
Hence, \eqref{eq:Phi1} is an equation for $\Phi$ with
velocity field given by \eqref{eq:u} and $\om$ given by
\eqref{eq:om2}.

Consider the operator $\cG$ formally defined by
\beq
\cG[\Phi](\bx,t):= \bx+ \int_0^t \bu(\Phi(\bx,s),s)~ ds
\eeq
with $\bu$ defined as in \eqref{eq:u} and $\om$ defined by \eqref{eq:om2}.

We obtain the following equation for $\Phi$:
\begin{align*}
\cG[\Phi] = \Phi.
\end{align*}
The local existence proof is now performed in a rather standard fashion, by finding a suitable metric space of flow maps, on which $\cG$ is a contraction. For the reader's convenience, a detailed proof can be found in the Appendix.

\subsection{Continuation of solutions.} 
We address the continuation of smooth solutions. As in \cite{MajdaBertozzi}, it can be shown that  
\begin{align*}
\int_0^T \mn \nabla \bu(\cdot, s) \mn_{L^\infty(\D)}~ ds < \infty 
\end{align*}
is sufficient to continue the solution to a slightly larger time interval (recall that $\mn \cdot \mn$ denotes a suitable $2\times 2$ matrix norm). 

Similar calculations as in the proof of Lemma \ref{lem:ubounded} show that $\mn \nabla \bu(\cdot, s) \mn$ can be bounded by a finite constant
provided $\sup_{t\in [0, T)} \|\om(\cdot, t)\|_\infty < \infty$ and
there exist $n_1, n_2 > 0$ such that
\begin{align}\label{eq:cond_continuation2}
\supp \om(\cdot, t) \subseteq [n_1, n_2]\times[0,\infty) 
\end{align}
for all $t\in [0, T)$.

Our goal here is to show that the solution can be continued if condition \eqref{cond_continuation} holds. In the following, fix some $T>0$ for which \eqref{cond_continuation} is true. 

In the first step, we show a bound on $\|\om\|_\infty$. Since the initial
data is nonnegative, all particles move to the left. Let $\bX(t)$ denote any particle trajectory such that $X_1(t) \geq k$. Then $X_1(s) \geq k$ at all times $s \leq t$, and by integrating the vorticity equation from \eqref{model} we obtain the a-priori bound
\begin{align*}
|\om(\bX(t), t)|&\leq \|\om_0\|_\infty + |\rho_0(\bX(0))|\int_0^T \frac{ds}{X_1(s)}\leq \|\om_0\|_\infty + \|\rho_0\|_\infty T k^{-1}.
\end{align*}
Together with \eqref{cond_continuation}, we obtain 
\begin{align}\label{eq:intom}
\int_0^T \|\om(\cdot, s)\|_\infty ~ds < \infty.
\end{align}
We now seek a function $n_1(t)$ and a constant $n_2$ such that 
\begin{align}\label{eq_supp_between_n1n2}
\supp \om(\cdot, t) \subseteq (n_1(t), n_2]\times[0,\infty)
\end{align}
for all $t\in [0, T)$. We may choose $n_2$ such that the support of
$\om_0$ lies to the left of $x_1 = n_2$. The support of $\om(\cdot, t)$ will remain to the left of $x_1 = n_2$ for all times since $u_1(\bx, t) = 0$ to the right of $x_1 = n_2$.
\begin{lemma}\label{lem:compare_velocities_n1}
Let $n_1(0)$ be such that 
\begin{align}\label{eq:initial_supp}
\supp \om(\cdot, 0) \subseteq (n_1(0), n_2]\times[0,\infty) 
\end{align}
and suppose \eqref{eq_supp_between_n1n2} holds on some time interval $[0, \hat T), \hat T < T$. Then
\begin{align}
-u_1((n_1(t),x_2), t)&\leq  \|\om(\cdot, t)\|_\infty n_1(t) \log\left(\frac{n_2}{n_1(t)}\right)
\end{align}
for $t\in [0, \hat T)$ and $x_2 \geq 0$.
\end{lemma}
\begin{proof}
Using \eqref{eq_supp_between_n1n2},
\begin{align*}
-u_1((n_1(t), x_2), t) &\leq \|\om(\cdot, t)\|_\infty\int_{n_1(t)}^{n_2} \int_{x_2}^\infty \frac{y_1 y_2~d\by}{(y_1^2 + y_2^2)^2}\leq \|\om(\cdot, t)\|_\infty\int_{n_1(t)}^{n_2} y_1^{-1} dy_1\\
& \leq \|\om(\cdot, t)\|_\infty \log\left(\frac{n_2}{n_1(t)}\right).
\end{align*}
\end{proof}
Now choose $n_1(t)$ to be a solution of $\dot n_1(t) = -2\|\om(\cdot, t)\|_\infty n_1(t) \log\left(\frac{n_2}{n_1(t)}\right)$ with a suitable $n_1(0)$ such that \eqref{eq:initial_supp} holds. Integrating this differential equation and using \eqref{eq:intom} yields $n_1(T) > 0$. From Lemma \ref{lem:compare_velocities_n1}, we see that 
the support of $\om(\cdot, t)$ lies to the right of $x_1 = n_1(t)$ on
the whole time interval $[0, T)$. Thus, by \eqref{eq:cond_continuation2}, we see that continuation is possible. 

\section{Construction of Singular Solutions}\label{sec_sing}

In the following, we consider smooth solutions that satisfy a number of \emph{control conditions} dependent on a set of \emph{parameters}. We seek to extend the validity of the control conditions up to the time of blowup, provided certain conditions were valid at $t=0$ and provided the parameters in Definition \ref{def:D0} are suitably chosen.

Let $g$ be the function defined by
\begin{align*}
g(x_1, t) = 2\alpha(x_1 - b(t))
\end{align*}
where $b(t)$ will be chosen below. 
We also define the following region 
\begin{align}\label{triangRegion}
D_t = \{ 0 \leq x_1 \leq \de,~0 \leq x_2 \leq g(x_1, t)\}
\cup \{ \de \leq x_1 \leq 4,~0\leq x_2 \leq g(x_1, 0) \},
\end{align}
where $\de>0$ will be chosen below (see Figure \ref{f2}).

For the remainder of the paper, $(\om, \rho)$ denotes a smooth
solution of \eqref{model} with suitably prepared
initial data in the sense of \eqref{cond_initial_data}.

\textbf{Observation.} Since $\rho_0, \om_0 \geq 0$, $\om$ remains 
non-negative for all times. This follows directly from the first equation of \eqref{model}. As a consequence, we note that 
\begin{align*}
\begin{split}
u_1(\bx, t) \leq 0,\\
u_2(\bx, t) \geq 0.
\end{split}
\end{align*}
So the particles all move to the left and upwards. Moreover, 
\begin{align*}
\supp~\om(\cdot, t) \subseteq (0, 4)\times[0, \infty)
\end{align*}
as long as the smooth solution is defined.

\begin{definition}
We call $(\om, \rho)$ controlled on $[0, T)$ if the following holds:
\begin{align}\label{controlConditions}
\begin{array}{lllr}
\phi_0 x_1^{-p_0} &<& \om(\bx, t) &\quad \bx\in D_t\cap \{x_1 < 1\},\\
\phi_0 &<& \om(\bx, t) &\quad \bx \in D_t \cap \{1\leq x_1 < 2\},\\
\om(\bx, t) &<& \phi_1 x_1^{-p_1}&\quad \bx\in D_t\cap \{x_1 < 1\},\\
\om(\bx, t) &<& \phi_1 &\quad \bx \in D_t\cap \{1\leq x_1<4\}.
\end{array}
\end{align}
for all $t\in [0, T)$.
\end{definition}

\begin{lemma} \label{lemma_estimates_Q} Suppose $(\om, \rho)$ is controlled on the time interval $[0, T)$.
Then 
\beq\label{eq:Qest1}
\begin{split}
Q(\bx,t) \ge \phi_0 x_1^{-p_0} F_0\left(\frac{x_2}{x_1},x_1, \alpha,p_0\right) \quad (\bx \in D_t, x_1 \leq 1)
\end{split}
\eeq
and 
\beq\label{eq:Qest2}
Q(\bx,t) \le \phi_1 x_1^{-p_1} F_1\left(\frac{x_2}{x_1},x_1,\alpha, p_1\right)\quad (\bx \in D_t),
\eeq
where
\begin{align*}
F_0&:=\til \cF_0+\cF_0
\end{align*}
with
\begin{align*}
 \cF_0(\eta, x_1, \alpha, p_0) &:= \int_{1}^{1/x_1} \si^{-p_0+1} G(\si, \alpha, \eta)~d\si,\\
\til \cF_0(\eta,x_1, \alpha, p_0) &:= x_1^{p_0} \int_{1/x_1}^{2/x_1} \sigma G(\sigma, \alpha, \eta)~ d\si
\end{align*}
and
\begin{align*}
F_1(\eta, x_1, \alpha, p_1) &:= \int_1^{1/{x_1}} \si^{-p_1+1} G(\si, \alpha, \eta)~ d\si+C.
\end{align*}
In the above, 
\begin{align*}
G(\si, \alpha, \eta) := \half\left[\frac{1}{\si^2+\eta^2} - \frac{1}{\si^2+(\alpha(\si-1)+\eta)^2}\right]
\end{align*}
and $C>0$ is a suitable universal constant. 
\end{lemma}

\begin{proof}
We note first that $\bx+S_\alpha\cap\{x_1 \leq 2\}$ is contained in $D_t$ for each $\bx \in D_t \cap \{x_1\leq 1\}$. 
For such $\bx = (x_1, x_2)$,
\begin{align*}
Q(\bx,t) &\ge \phi_0 \int_{x_1}^1 \int_{x_2}^{\alpha(y_1-x_1)+x_2} \frac{y_1^{1-p_0} y_2}{(y_1^2+y_2^2)^2} ~dy_2dy_1 +\phi_0 \int_1^2 \int_{x_2}^{\alpha(y_1-x_1)+x_2} \frac{y_1 y_2}{(y_1^2+y_2^2)^2}~dy_2dy_1\\
&=\half \phi_0 \int_{x_1}^1  y_1^{-p_0+1} \left[ \frac{-1}{y_1^2+y_2^2}\right]_{x_2}^{\alpha(y_1-x_1)+x_2}~dy_1+\frac{1}{2} \phi_0 \int_{1}^2  y_1 \left[ \frac{-1}{y_1^2+y_2^2}\right]_{x_2}^{\alpha(y_1-x_1)+x_2}~dy_1\\
&=\half\phi_0 \int_{x_1}^1  y_1^{-p_0+1} \left[\frac{1}{y_1^2+x_2^2} - \frac{1}{y_1^2+(\alpha(y_1-x_1)+x_2)^2}\right]~dy_1
\\
& \quad+ \half \phi_0 \int_1^2  y_1 \left[\frac{1}{y_1^2+x_2^2} - \frac{1}{y_1^2+(\alpha(y_1-x_1)+x_2)^2}\right]~dy_1\\
&=:I_1+I_2.
\end{align*}
$I_1, I_2$ can be written as
\begin{align*}
I_1&\stackrel{y_1=\si x_1}{=}\phi_0 x_1^{-p_0}\half  \int_{1}^{1/x_1} \si^{-p_0+1} \left[\frac{1}{\si^2+(x_2/x_1)^2} - \frac{1}{\si^2+(\alpha(\si-1)+(x_2/x_1))^2}\right]~d\si,\\
I_2&\stackrel{y_1=\si x_1}{=}\half \phi_0 \int_{1/x_1}^{2/x_1}  \si \left[\frac{1}{\si^2+(x_2/x_1)^2} - \frac{1}{\si^2+(\alpha(\si-1)+(x_2/x_1))^2}\right]~d\si,
\end{align*}
yielding the representation \eqref{eq:Qest1}.

For the upper bound, we use the upper bounds from the control condition and find
\begin{align*}
Q(\bx,t) &\le \phi_1 \int_{x_1}^1 \int_{x_2}^{\alpha(y_1-x_1)+x_2} \frac{y_1 y_2}{(y_1^2+y_2^2)^2} y_1^{-p_1}~d\by
+ \phi_1 \int_{1}^4 \int_{x_2}^{\alpha(y_1-x_1)+x_2} \frac{y_1 y_2}{(y_1^2+y_2^2)^2} ~d\by \\
&\le \frac{1}{2}\phi_1\int_{x_1}^1 y_1^{1-p_1} \left[ \frac{-1}{(y_1^2+y_2^2)}\right]_{x_2}^{\alpha(y_1-x_1)+x_2} dy_1
+\frac{1}{2}\phi_1\int_{1}^4 y_1\left[ \frac{-1}{(y_1^2+y_2^2)}\right]_{x_2}^{\alpha(y_1-x_1)+x_2} dy_1 \\
&=: \til I_1 + \til I_2.
\end{align*}
For $\til I_1$ we have
\begin{align*}
\til I_1 &= \frac{1}{2}\phi_1\int_{x_1}^{1} y_1^{1-p_1} \left[\frac{1}{y_1^2+x_2^2} - \frac{1}{y_1^2+(\alpha(y_1-x_1)+x_2)^2}\right]~dy_1\\
& \stackrel{y_1=\si x_1}{=} \phi_1 x^{-p_1} \int_{1}^{1/x_1} G(\si, \alpha, x_2/x_1)~ d\si,
\end{align*}
whereas for $\til I_2$ we have
\begin{align*}
\til I_2 \leq \half \phi_1 \int_1^4 \frac{1}{y_1}~dy_1 \leq C \phi_1.
\end{align*}
\end{proof}

We now fix the choice of $b(t)$:

\begin{definition} Let $b$ be the solution of
\begin{align}\label{eq_for_b}
\begin{split}
\dot b(t) &= - b(t)^{1-p_0} \phi_0 \cF_0\left(0, \half, \alpha, p_0\right)\\
b(0) &= b_0.
\end{split}
\end{align}
Note that $(b(t),0)$ does not correspond to any particle trajectory.
\end{definition}
As already mentioned, the main idea of the proof is to show that the solution is controlled on the time-dependent control region
$D_t$ up to the blowup time. In the following, we will use the notation $\bX(t) = (X_1(t), X_2(t))$ for particle trajectories
\begin{align*}
\frac{d \bX}{d t}(t) = \bu(\bX(t), t),
\end{align*}
with initial position $\bX(0)$.
In particular, we need information on the initial positions of the particles that are inside the control region $D_t$ at any time $t > 0$. The following Lemmas achieve this. In order to avoid 
constant  repetitions, we state the following

\textbf{General Hypothesis.} In Lemmas \ref{lemma_that_prevents_crossing1}-\ref{lem:T1}, the solutions is assumed to be controlled up to some time $T_c > 0$.

\begin{lemma}\label{lemma_that_prevents_crossing1}
Suppose 
\begin{align}\label{contradiction1}
 \cF_0\left(0,\frac{1}{2}, \alpha, p_0\right)- \cF_0(\eta,x_1, \alpha, p_0) < \frac{\eta}{2\al}\cF_0(\eta, x_1, \alpha, p_0)
\end{align}
holds for all $\eta\in [0, 2\al]$ and $x_1\in [0,2\de]$ for some $0 < \de < \frac{1}{4}$. 
Consider all particles $\bX(t)$ with
\begin{align*}
b_0 &\le X_1(0),~~ g(X_1(0), 0) < X_2(0)
\end{align*}
(i.e. those that start at time $t=0$ above the graph of $g(\cdot, 0)$). In the course of the evolution, these particles cannot enter the region $D_t$ at a point $\bx^* = (x_1^*, x_2^*)$ with $b(t) < x_1^* < 2\de$ and at some time $t < T_c$.
\end{lemma}
\begin{proof}
Let $\bX(t)=(X_1(t),X_2(t))$ be a trajectory as above that crosses $g(\cdot, t^*)$ at some time $t^*$ (see Figure \ref{f2} (c)). We will derive a contradiction.
\begin{figure}[htbp]
\hspace{-1cm}
\includegraphics[scale=0.3]{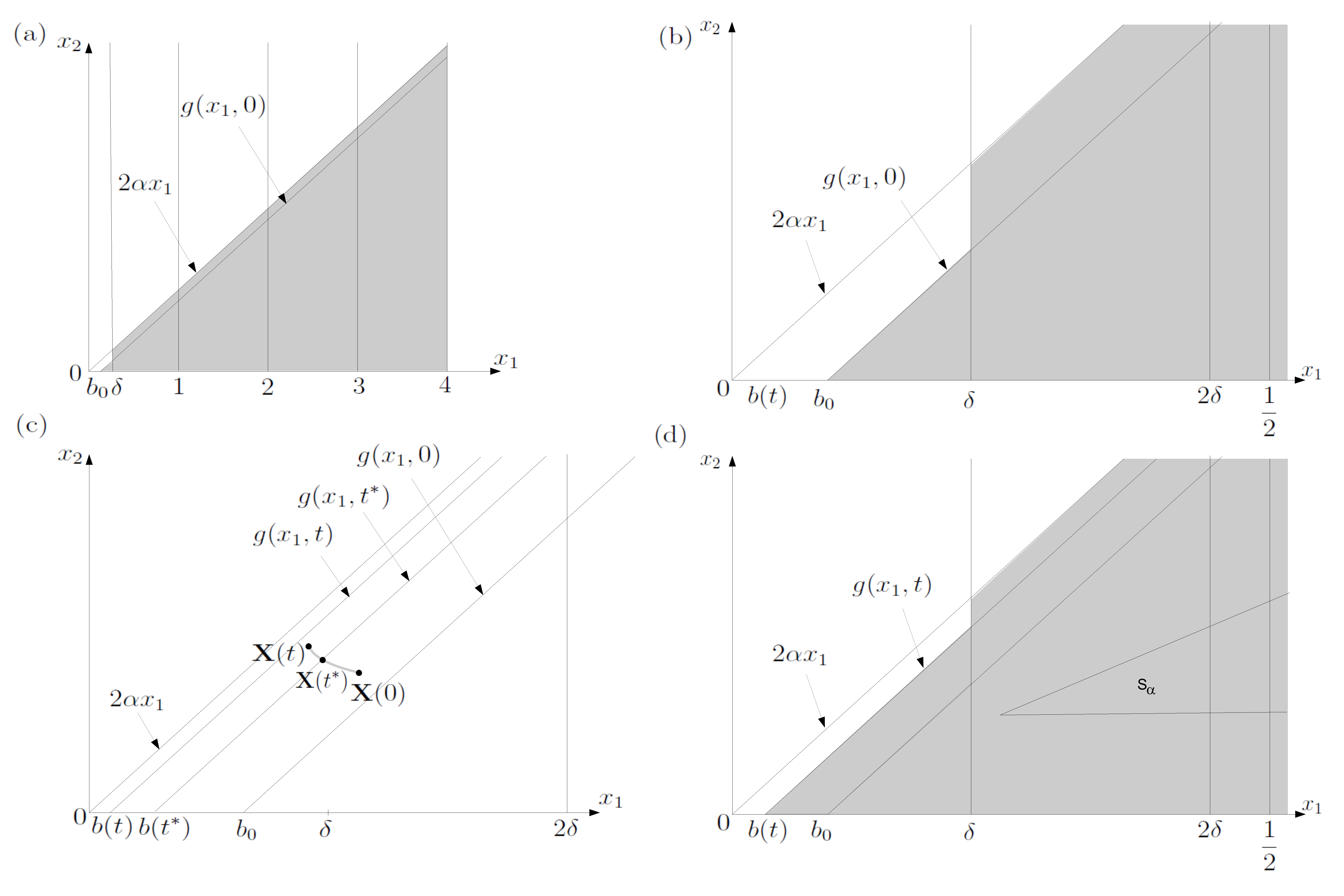}%
\caption{(a) The gray area represents $D_0$. (b) A closeup of $D_0$ near the origin. (c) The scenario prevented by Lemma \ref{lemma_that_prevents_crossing1}:  Crossing of $\bX$ and $g$ at $t=t^*$. The gray line represents the trajectory $\bX$ between $0$ and $t$. (d) A closeup of $D_t$ and the sector $S_\al$. \label{f2}}
\end{figure}
Let 
$$f(t):=X_2(t)-g(X_1(t),t)$$
and observe that $f(0) > 0$. Assume crossing happens, i.e. 
$f(t^*) = 0, \displaystyle\frac{d f}{dt}(t^*) \le 0$ at a point $\bX(t^*)$ such that $b(t^*) < X_1(t^*) < 2\delta$.
We write $\bX=\bX(t^*)=(X_1, X_2)$, and in the following all expressions are evaluated at time $t^*$. We skip $t^*$ and suppress the nonessential arguments $\al, p_0$. First note
\begin{align}\label{eq:x2dot}
\dot X_2 &= X_2 Q(\bX, t) \ge \phi_0 X_2 X_1^{-p_0}\cF_0(X_2/X_1, X_1).
\end{align}
A computation gives 
$\displaystyle\frac{df}{dt} = \dot X_2 - 2\al(\dot X_1 - \dot b)$
and from \eqref{eq:x2dot}, the assumption $\displaystyle\frac{df}{dt}(t^*) \le 0$, and the definition of $b(t)$ follows
\begin{align*}
X_2 \phi_0 X_1^{-p_0}\cF_0\left(\frac{X_1}{X_2},X_1\right) &\le \dot X_2 \le 2\al(\dot X_1 - \dot b)=2\al(-X_1Q - \dot b)\\
&\le 2\al\phi_0\left[-X_1^{1-p_0}\cF_0\left(\frac{X_2}{X_1},X_1\right)+b^{1-p_0}\cF_0\left(0,\half\right)\right]\\
&= 2\al\phi_0\left[X_1^{1-p_0}\left(\cF_0\left(0,\half\right)-\cF_0\left(\frac{X_2}{X_1},X_1\right)\right)+(b^{1-p_0}-X_1^{1-p_0})\cF_0\left(0,\half\right)\right]\\
&\le 2\al\phi_0X_1^{1-p_0}\left(\cF_0\left(0,\half\right)-\cF_0\left(\frac{X_2}{X_1},X_1\right)\right)
\end{align*}
because $(b^{1-p_0}-X_1^{1-p_0})\cF_0\left(0,\half\right)\le 0$. 
The above estimation implies
\begin{align}\label{Lemma1_eq1}
(X_1-b)\cF_0\left(\frac{X_2}{X_1},X_1\right)&\le X_1\left[\cF_0\left(0,\half\right)-\cF_0\left(\frac{X_2}{X_1},X_1\right)\right]. 
\end{align}
Now write $\eta = X_2/X_1$, and $X_1 - b = X_2/2\al$. \eqref{Lemma1_eq1} becomes
\begin{align}\label{Lemma1_eq2}
\frac{\eta}{2\al}\cF_0\left(\eta,X_1\right)&\le \left[\cF_0\left(0,\half\right)-\cF_0\left(\eta,X_1\right)\right]. 
\end{align}
with some $\eta=X_2/X_1=2\al
\left(1-\frac{b}{X_1}\right)\in [0,2\al]$
We arrive at a contradiction since \eqref{contradiction1} implies that \eqref{Lemma1_eq2} does not hold for $x_1 < 2\delta$. 
\end{proof}
\begin{remark}
In Lemma \ref{lemma_that_prevents_crossing2} we will show that the condition \eqref{contradiction1}
holds if $\de$ and $p_0$ are suitably chosen. 
\end{remark}

\begin{corollary}\label{cor}
Let the assumptions of Lemma \ref{lemma_that_prevents_crossing1} hold. Let $t\in [0,T_c)$. Then
for any particle trajectory $\bX(t)$ with $X_1(0)<4$ such that
$$
\bX(t) \in D_t 
$$
we have $\bX(\tilde t)\in D_{\tilde t}$ for all $\tilde t \leq t$.
\end{corollary}
\begin{proof}
For such a particle trajectory, define the set 
$$
O(\bX(t)) = \{\bx: X_1(t)\leq x_1 \leq 4, ~~x_2 \leq X_2(t)\}.
$$
We have $\bX(\tilde t)\in O(\bX(t))$ for all $\tilde t < t$, since all the particles move to the left and up. If $X_1(t)\geq \de$, there is nothing to prove, as $O(\bX(t))$ is contained in $D_{\tilde t}$ for all $\tilde t\leq t$.

So suppose $X_1(t)< \de$. If in this case there is a
$\tilde t < t$ with $\bX(\tilde t) \notin D_{\tilde t}$, then $\bX(\tilde t)$ lies above $g(\cdot, \tilde t)$ and the trajectory must have crossed $g(\cdot, t^*)$ at some time $\tilde t < t^*\leq t$. Use $x^* = (x_1^*, x_2^*)$ to denote
the crossing point. $X_1(\tilde t)$ lies to the left of $x_1 =\de$, as does the the crossing point $x^*$.
Hence the crossing is excluded by Lemma 
\ref{lemma_that_prevents_crossing1} (see Figure \ref{f2} (c)).
\end{proof}

\begin{lemma}\label{lemma_that_prevents_crossing2} 
For fixed $\al\in(0, \infty)$, there exists a $0<\de<\frac{1}{4}$ such that \eqref{contradiction1} is true for all $\eta\in [0, 2\al], x_1\in [0, 2\de]$ and all sufficiently small $p_0>0$.
\end{lemma}
\begin{proof}
We have, noting that $1/x_1 > 2$,
\begin{align*}
\cF_0(0,1/2)-\cF_0(\eta,x_1)&\le \frac{1}{2}\int_{1}^{1/x_1}\sigma^{-p_0+1}\left[\frac{\eta^2}{\sigma^2(\sigma^2+\eta^2)}+\frac{-2\alpha (\si-1) \eta - \eta^2}{(\sigma^2+\al^2(\sigma-1)^2)(\si^2+(\al(\si-1)+\eta)^2)}\right]d\sigma\\
&\leq\frac{\eta^2}{2}\int_{1}^{\infty}\sigma^{-p_0+1}\frac{d\sigma}{\sigma^2(\sigma^2+\eta^2)}.
\end{align*}
So \eqref{contradiction1} is implied by
\begin{align*}
\frac{\eta^2}{2}\int_{1}^{\infty}\frac{\sigma^{-p_0+1}~d\sigma}{\sigma^2(\sigma^2+\eta^2)} &< 
 \frac{\eta}{4\alpha}  \int_{1}^{1/x_1}  \si^{-p_0+1} \frac{\alpha^2(\sigma-1)^2+2\alpha\eta(\sigma-1)}{(\si^2+\eta^2)(\si^2+(\alpha(\si-1)+\eta)^2)}d\sigma.
\end{align*}
Since $\eta \leq 2\alpha$, a sufficient condition for the last
inequality to hold for all $x_1 \leq 2\de$ is in fact 
\begin{align*}
(2\alpha)^2 \int_{1}^{\infty}\frac{\sigma^{-p_0+1}~d\sigma}{\sigma^2(\sigma^2+\eta^2)} &< 
   \int_{1}^{1/(2\de)}  \si^{-p_0+1} \frac{\alpha^2(\sigma-1)^2}{(\si^2+\alpha^2)(\si^2+\alpha^2\si^2)}d\sigma.
\end{align*}
As $p_0 \to 0$, the integral on the left-hand side converges to some
positive number, whereas the integral on the right-hand side is arbitrarily large as $p_0 \to 0$ and for all sufficiently small $\de>0$. This
finishes the proof of the Lemma. 
\end{proof}

\begin{corollary}(Upper and lower bounds for $F_0,F_1$)\\
For $\eta \in [0,2\al]$, we have the following bounds
\begin{align}
\begin{split}
F_1(\eta, \alpha, p_1)\leq M_1(\alpha)\\
M_0(\alpha) \leq F_0(\eta, \alpha, p_0) 
\end{split}
\end{align}
\end{corollary}
\begin{proof} 
The upper bound for $F_1$ is found as follows:
\begin{align*}
F_1(\eta, \alpha, p_1) &\le   \int_{1}^{\infty}  \si^{-p_1+1} \frac{\alpha^2(\sigma-1)^2+2\alpha\eta(\sigma-1)}{(\si^2+\eta^2)(\si^2+(\alpha(\si-1)+\eta)^2)}d\sigma+C\\
&\le   \int_{1}^{\infty} \si^{-p_1+1} \frac{(\alpha^2(\sigma-1)+4\alpha^2)(\si-1)}{(\si^2+\eta^2)(\si^2+(\alpha(\si-1)+\eta)^2)}d\sigma+C\\
&\le   \int_{1}^{\infty} \si^{\frac{1}{2}} \frac{(\alpha^2(\sigma-1)+4\alpha^2)(\si-1)}{\si^2(\si^2+\alpha^2(\si-1)^2)}d\sigma + C=: M_1(\alpha)
\end{align*}
where we have used $\eta \le  2\alpha$ and $\frac{1}{2}\le p_1$. Note that the last integral appearing in the estimate is convergent. 

Turning now to the lower bound, we compute for $x_1 \leq \frac{1}{2}$
\begin{align*}
F_0(\eta, x_1, \al, p_0) &\geq \cF_0(\eta, x_1, \al, p_0) 
\geq  \int_{1}^{2} \si^{-p_0+1}G(\si, \al, \eta)~d\si\geq  \int_{1}^{2} \si G(\si, \al, \eta)~d\si
\end{align*}
where we have used $p_0 > 0$.
For $\frac{1}{2}\leq x_1 \leq 1$, we see that
\begin{align*}
F_0(\eta, x_1, \al, p_0) &\geq \til\cF_0(\eta, x_1, \al, p_0) \geq  \min_{x\in [\frac{1}{2}, 1]}\int_{1/x_1}^{2/x_1} \si G(\si, \al, \eta)~d\si.
\end{align*}
So we can take 
\begin{align*}
M_0(\alpha) = \min\left\{ \int_{1}^{2} \si G(\si, \al, \eta)~d\si, 
 \min_{x\in [\frac{1}{2}, 1]}\int_{1/x_1}^{2/x_1} \si G(\si, \al, \eta)~d\si \right\}.
\end{align*}
\end{proof}

\begin{lemma}\label{lem:T1}
Let $0<T_1<1$ satisfy
\begin{align*}
(\cM+1) T_1 < \frac{1}{2}
\end{align*}
where 
$$
\cM = \max\{\om_0(\bx) : x_1 \geq 1\}.
$$
Consider particle trajectories $\bX(t)$ with $X_1(0)> 3$. Then for all
$t \leq T_1$, we have $X_1(t) \geq \frac{5}{2}$.
\end{lemma}

\begin{proof}
First we derive a vorticity estimate for all $t\leq T_1$ and $\bx\in \D$, $x_1 \geq 2$. Let $\bX(t)$ be a particle trajectory such that $\bX(t) = \bx$.
Recall $u_1(\bx, t)\leq 0$ for 
all $\bx$. As a consequence, all particles with $X_1(t) > 2$ have
the property that $X_1(s) > 2$ for all times $0\leq s < t$. Integrating the
vorticity equation along a particle trajectory and using
$X_1(s)\geq 2, \rho_0(\bx) \leq 1$, we obtain the estimate
\begin{align*}
\om(x_1, x_2, t) \leq \|\om_0\|_\infty + T_1
\end{align*}
for all $(x_1, x_2)$, $x_1 > 2$, $t\leq T_1$. Using again that particles move to the left,
we have $\supp \om(\cdot, t) \subseteq (0, 4)\times[0, \infty)$. Hence
we estimate the velocity $u_1$ for all $x_1 \geq 2, t\leq T_1$:
\begin{align*}
-u_1(\bx, t) &\leq x_1\iint_{\bx+S_\al} \frac{y_1y_2}{|\by|^4} \om(\by,t)~d\by
\le (\cM +  T_1) \int_{2}^4 \int_{x_2}^\infty
\frac{y_1 y_2}{|\by|^4}~dy_2 dy_1\\ 
&\le (\cM +  T_1) \int_2^4 \frac{y_1}{y_1^2+x_2^2}~dy_1 \leq 
(\cM + T_1).   
\end{align*}
From the above estimate for the velocity follows that for particles with $X_1(0) > 3$,
\begin{align*}
X_1(t) \geq 3 -  (\cM + T_1) T_1 > \frac{5}{2}
\end{align*}
by choice of $0<T_1<1$.
\end{proof}

\begin{lemma}\label{lem:T*}
The function $t\mapsto b(t)$ is strictly decreasing and we have $b(T^*) = 0$, where $T^*>0$ is given by
\begin{align}\label{def_Tstar}
T^*=\frac{b_0^{p_0}}{p_0 \phi_0 \cF(0, \half, \alpha, p_0)}.
\end{align}
If the solution $(\om, \rho)$ is controlled on $[0, T^*)$, then $T_s \leq T$. 
\end{lemma}

\begin{proof}
It follows from the definition of $b(t)$ that $b(t)$ is strictly decreasing.
Integrating the differential equation \eqref{eq_for_b}, we obtain
\begin{align*}
b(t)^{p_0} = b_0^{p_0} - p_0 \phi_0 \cF_0(0, \half, \alpha, p_0) t,
\end{align*}
so $b(T^*) = 0$. 

Now assume that the solution is controlled on $[0, T^*)$. Using Lemma
\ref{lemma_estimates_Q}, we can compare the particle velocities with $\dot b(t)$.
Namely, we first observe that for all $x_1\leq \half, t<T^*$,
\begin{align*}
- u_1(x_1, 0, t) \geq x_1^{1-p_0}\cF_0(0, x_1, \al, p_0) \geq x_1^{1-p_0} \cF_0(0, \half, \al, p_0).   
\end{align*}
This means that for all particles with $X_1(0) < b_0$, $X_1(t) < b(t)$ holds for 
$t < T^*$. Hence the solution cannot remain smooth past $T^*$, since otherwise there are
particle trajectories that collide with $\bx = (0, 0)$.
\end{proof}

\begin{theorem}\label{theorem_Control}
Let
\begin{align*}
T_1 &= \frac{1}{2(\cM + 1)}\\
T_2 &= \phi_1 - \cM
\end{align*}
and assume $\phi_1$ is such that $T_2 > 0$.

Suppose the following conditions are all satisfied for the positive numbers
$\phi_0, \phi_1, 0< p_0 < \frac{1}{2}, p_1 = 1-p_0, b_0$:
\begin{align}\label{cond1} 
\frac{1}{M_0(\alpha) (1-p_0)} &< \phi_0 \phi_1 < \frac{1}{M_1(\alpha)(1-p_1)},\\
\label{cond2} T^{*} &= \frac{b_0^{p_0}}{p_0 \phi_0 M_0(\alpha)} < \min\{T_1, T_2\}
\end{align}
Moreover, assume that $\de>0, p_0>0$ are such that the conclusion of Lemma \ref{lemma_that_prevents_crossing2} holds.
Then $(\om, \rho)$ will be controlled on $[0, T_s)$, $T_s>0$ denoting the maximal lifespan of the solution. Moreover, $T_s < \infty$.
\end{theorem}

\begin{proof}
Since the initial data is suitably prepared, there exists a time interval $[0, \tau), \tau>0$ on which the solution is controlled. We define $T_c>0$ to be the supremum
of all times $T \leq T_s$ such that the solution is controlled on
$[0, T)$.
\begin{itemize}
\item As an important preliminary observation, we note that for the particle trajectories $\bX(t)$ that are in $D_t$ at some time $t < T_c$, the following holds. If $X_1(0) < 4$, then
those particles originate from the region $D_0$ and are in
$D_t$ for all times $t$ prior to $T_c$.
This follows from Corollary \ref{cor}.
\end{itemize}
Now suppose that 
\begin{align}\label{eq:hypo}
T_c < \min\{T_s,T^{*}\}
\end{align}
either the upper or the lower bound of  \eqref{controlConditions} fails at $t = T_c$. E.g. if the upper bound fails, then
$X(T_c) \geq \phi_1 X_1^{-p_1}(T_c)$ for some particle trajectory $\bX(t)$.
We exclude this by tracking the evolution of particles $X(t)$ such that $X(T_c) \in D_{T_c}$. We distinguish
several cases:

\emph{Case 1:} $X_1(T_c) \leq 1$. If $X_1(0)> 1$, we let 
$\til t \leq T_c$ be the time such that $X_1(\til t) = 1$.
Otherwise, we let $\til t = 0$. First note that $X_1(0)>3$,
since by the assumption \eqref{eq:hypo}, we have $T_c < T_1$
and particles with $X_1(0) > 3$ do not have enough time to
cross into $\{x_1 < 1 \}$ by Lemma \ref{lem:T1}. 

So, by the basic observation above we know that the particle $\bX(t)$ was inside $D_{t}$ for all $0\leq t < T_c$.
We can estimate as follows, using $\rho(\bx, t)\leq 1$:
\begin{align*}
\om(\bX(t), t) &\leq \om(\bX(\til t), \til t) + \int_{\til t}^t \frac{ds}{X_1(s)}= \om(\bX(\til t), \til t) + \int_{\til t}^t X_1(s)^{-1} \frac{(-\dot X_1(s))}{(-\dot X_1)(s)}~ds\\
&\leq  \om(\bX(\til t), \til t) + \frac{1}{M_0 \phi_0} \int_{\til t}^t X_1(s)^{-2+p_0} (-\dot X_1(s)) ~ds\\
& =  \om(\bX(\til t), \til t) + \frac{1}{M_0 \phi_0 (1-p_0)}  \left(X_1(t)^{-1+p_0} - X_1(\til t)^{-1+p_0}\right) .
\end{align*}
Here, we use the fact that the solution is controlled up to time $T_c$ and $X_1(t) \leq 1$ for $t\in [\tilde t, T_c)$, so we are allowed to use
$$-\dot X_1(t) \geq \phi_0 M_0(\alpha)X_1^{1-p_0}(t),$$ 
which follows from Lemma \ref{lemma_estimates_Q}.

The foregoing implies that 
\begin{align*}
\om(\bX(t), t) \phi_0 X_1(t)^{p_1} &\leq \om(\bX(\til t), \til t))\phi_0 X_1(t)^{p_1} +
\frac{1}{M_0 (1-p_0)}  \left(1 - X_1(\til t)^{-1+p_0} X_1(t)^{p_1}\right)\\
&= \om(\bX(\til t), \til t))\phi_0 X_1(\til t)^{p_1} \mu   +
\frac{1}{M_0 (1-p_0)}  \left(1 - \mu\right)
\end{align*}
where $\mu = (X_1(t)/X_1(\til t))^{p_1}$ and we have used $p_0+p_1 = 1$. So $\om(\bX(t), t)< \phi_1 X_1(t)^{-p_1}$ is implied by the two inequalities
\begin{align}\label{eq:two_cond1}
\begin{split}
\om(\bX(\tilde t), \tilde t)& < \phi_1 X_1(\til t)^{-p_1}\\
\frac{1}{M_0 (1-p_0)} &< \phi_0 \phi_1.
\end{split}
\end{align}
To see that the first one of the preceding inequalities is true, we distinguish the
cases $\til t = 0$ and $\til t > 0$. 
If $\til t > 0$, it follows from the the fact that $\om$ was controlled at time $\til t$. In case $\til t = 0$, $\om(\bX(\til t), \til t) = \om_0(\bX(0))$. Using the basic observation above, we note that $\bX(0)$ lies in $D_0$ and so the first line of \eqref{eq:two_cond1} follows from \eqref{cond_initial_data}. The second line of \eqref{eq:two_cond1} follows from \eqref{cond1}.

By a similar calculation (using that $\rho_0(\bX(0)) = 1$ for $\bX(0) \in D_0\cap \{x_1 < 3\}$), the inequality $\om(\bX(t), t)> \phi_0 X_1(t)^{-p_0}$,
is implied by the two conditions:
\begin{align*}
\begin{split}
\om(\bX(\til t), \til t) > \phi_0 X_1(\til t)^{-p_0}\\
\frac{1}{M_1 (1-p_1)} > \phi_0 \phi_1.
\end{split}
\end{align*}
Again, the first follows from the fact that the solution was either controlled
at $t= \til t$ or from the initial condition, and the second holds by assumption \eqref{cond2}. In summary, we have shown that
$$
\phi_0 X_1(T_c)^{-p_0} < \om(\bX(T_c), T_c) < \phi_1 X_1(T_c)^{-p_1},
$$
and the control conditions on $D_t\cap \{x_1 < 1\}$ are not violated.

\emph{Case 2:} $1 < X_1(T_c) \leq 2$. In this region, we fist show that the lower bound $\om(\bx, t) > \phi_0$ cannot be violated. By Lemma \ref{lem:T1} and \eqref{cond2}, $X_1(0) \leq 3$. Hence,
\begin{align*}
\om(\bX(T_c), T_c) \geq \om_0(\bX(0)) > \phi_0.
\end{align*}
because of $\rho_0 \geq 0$. Concerning the upper bound,
we note 
\begin{align*}
\om(\bX(T_c), T_c) &\le \om_0(\bX(0))+\rho(\bX(0))\int_0^{T_c} \frac{ds}{X_1(s)}\leq \cM +  T_c  
\end{align*}
where we used $X_1(s) \geq 1$. Since we are still working under the hypothesis $T_c < T^*$, $T_c < T_2$ holds and thus
$\om(\bX(T_c), T_c) < \phi_1$.

\emph{Case 3:} $3 < X_1(T_c) \leq 4$. In this case, the control condition does not contain a lower bound. The argument for the upper bound is 
is the same as in the previous case.

In summary, we have shown that \eqref{eq:hypo} does not
hold, i.e. $T_c = \min\{T_s, T^*\}$. So either $T_s < T^*$
in which case the solution stays controlled up to $T_s$,
or $T^* \leq T_s$ and hence $T_c = T^*$, so that by Lemma
\ref{lem:T*} we obtain $T_s\leq T^*$. This yields again the same conclusion.
\end{proof}

\begin{proof}[Proof of Theorem \ref{theorem_Blowup}]
In order to complete the proof of Theorem \ref{theorem_Blowup}, we need to show
that the conditions given in Theorem \ref{theorem_Control} can be satisfied by
choosing $p_0, p_1, \phi_0, \phi_1, \de, b_0$ and such that the set of initial data satisfying \eqref{cond_initial_data} is not empty.

First we observe that the inequality
\begin{align*}
\frac{1}{M_0(\alpha) (1-p_0)} < \frac{1}{M_1(\alpha) (1-p_1)}
\end{align*}
can be satisfied by choosing $p_1$ close to $1$, i.e. $p_0$ close to zero
(recall that $p_0 + p_1 = 1$). Note also that \eqref{contradiction1} holds for
sufficiently small $p_0 > 0$ (see Lemma \ref{lemma_that_prevents_crossing2}).

Next define $\phi_1$ for given $\phi_0$ by 
\begin{align*}
\phi_1 = \frac{1}{2\phi_0 }\left( \frac{1}{M_0(\alpha) (1-p_0)} + \frac{1}{M_1(\alpha) (1-p_1)} \right).
\end{align*}
With this choice of $p_0$ and $\phi_1$, \eqref{cond1} holds. 

Next choose $\phi_0>0$ sufficiently small, so that 
the set of initial data satisfying \eqref{cond_initial_data}
is not empty,
Up to now, $p_0, p_1, \phi_0, \phi_1, \de$ have been fixed. In order for \eqref{cond2} to hold, we just pick sufficiently small $b_0 > 0$.

The only statement that remains is to show that $\|\om(\cdot, t)\|_{L^\infty(\D)}\to \infty$ as $t\to T_s$. But since the initial data was assumed to be nonnegative, this is implied by the continuation criterion 
in Theorem \ref{thm_localExistence}.
\end{proof}

\section{Acknowledgements}
We would like to thank A. Kiselev for helpful discussions and suggestions, and a careful reading of a first version of the manuscript. VH expresses his gratitude towards S.~Denisov who, in a discussion, directed VH's attention to the idea that steady singular solutions may play a key role in controlling blowup solutions. VH acknowledges support by German Research Foundation grants HO 5156/1-1 and HO 5156/1-2, NSF grant DMS-1614797 and partial support by NSF grant DMS-1412023.

\section{Appendix: Local Existence and Uniqueness of Solutions}\label{app}
In order to find a metric space on which $\cG$ is defined, we consider flow maps that are perturbations of the identity map on the quadrant $\R_+^2$. 
\begin{definition}
Let $\cB$ be the set of all $\Phi\in C([0,T],C^1(\R^2_+,\R^2_+))$
such that the following properties are true:
\begin{align}
&\Phi(0, 0) = (0, 0), \label{prop_1} \\
\begin{split}\label{prop_2}
\Phi(\R_+\times \{0\}, t) \subseteq \R_+\times \{0\},\\
\Phi(\{0\}\times \R_+, t)  \subseteq \{0\}\times \R_+
\end{split}\\
&\Phi((0,\infty)^2, t)\subseteq (0,\infty)^2.\label{prop_3}
\end{align}
Moreover, $\Phi$ should be of the form
\begin{align}\label{prop_4}
\Phi=\Id+\hat\Phi, \quad \|\hat \Phi\|\le\zeta
\end{align}
where $\Id$ means the mapping $\Id(\bz,t)=\bz$ and
$$\|\hat\Phi\|:=\sup_{t\in [0,T]}\left(\sup_{\bz\in\D}|\hat\Phi(\bz,t)|+\sup_{\bz\in\D}\mn\nabla\hat\Phi(\bz,t)\mn \right).$$
$\cB$ is a complete metric space with metric
$$
d(\Id+\hat \Phi,\Id+\hat \Psi) = \|\hat \Phi-\hat \Psi\|.
$$
\end{definition}
The property that $\cB$ is a complete metric space uses the following
Lemma.

\begin{lemma}
For sufficiently small $\zeta > 0$, any $\Phi\in \cB$ and $t\in[0,T]$ 
$$\Phi(\cdot,t):\D\to\D$$
is a diffeomorphism of $(0, \infty)^2$ onto $(0, \infty)^2$ and moreover, 
$$
x_1\mapsto \Phi_1(x_1, 0, t), x_2\mapsto \Phi_2(0, x_2, t)
$$
are diffeomorphisms of $(0, \infty)$ onto $(0,\infty)$.
\end{lemma}

\begin{proof}
\begin{align*}
\nabla \Phi(\bz,t)&=\nabla(\Id +\hat \Phi)(\bz,t)=
I_{2\times 2}+
\begin{pmatrix}
\d_{z_1}\hat\Phi_1&\d_{z_2}\hat\Phi_1 \\ \d_{z_1}\hat\Phi_2&\d_{z_2}\hat\Phi_2
\end{pmatrix}(\bz,t)
\end{align*}
where $I_{2\times 2}$ denotes the $2\times 2$ identity matrix.
Since $\|\hat \Phi\|\le \zeta$ we have $\mn \nabla \hat \Phi(\bz,t)\mn\le \zeta$ for all $\bz\in \D$. It follows that for $\zeta < 1$,
$$\mn\nabla\Phi(\bz,t)^{-1}\mn\le \frac{1}{1-\zeta}$$ for all $\bz\in (0, \infty)^2$ and hence $\Phi$ is a local diffeomorphism everywhere in $(0, \infty)^2$.  
$\Phi\in \cB$ means by definition that $\Phi((0,\infty)^2, t)\subseteq (0, \infty)^2$. In order to show that $\Phi$ is a global diffeomorphism of $(0, \infty)^2$, it suffices to show that $\Phi$ is a proper map, i.e. the preimage of any compact set $\cK\subset (0,\infty)^2$ is compact in $(0, \infty)^2$. Observe that a compact set in $(0, \infty)^2$ is a closed, bounded subset of $(0, \infty)^2$ with positive distance to the axes. Using $|\Phi(\bx, t)| \to \infty$ as $|\bx| \to \infty$, we see $\Phi^{-1}(\cK)$ is bounded. Then \eqref{prop_2} implies that $\Phi^{-1}(K)$ has a positive distance to the axes.
Now applying a well-known theorem \cite{Gordon}, we may conclude that $\Phi$ maps $(0, \infty)^2$ onto $(0, \infty)^2$ and that $\Phi^{-1}:(0,\infty)^2\to(0,\infty)^2$ exists and is $C^1$. 

To show that e.g. $x_1\mapsto \Phi_1(x_1, 0, t)$ is 
a diffeomorphism on $(0, \infty)$, we first note that by
\eqref{prop_2}, $\Phi_1((0,\infty), 0, t)\subset [0,\infty)$. The derivative $\partial_{x_1}\Phi_1(x_1, 0, t)$ is given by
$1-\partial_{x_1}\hat \Phi_1(x_1, 0, t)$ and is also uniformly bounded away from zero for small $\zeta>0$.
\end{proof}

Using the preceding Lemma, one can show that for $\Phi\in \cB$, $\Phi(\bx,t)=(\Phi_1(\bx,t),\Phi_2(\bx,t))$, the expression $\cG[\Phi]$ is well-defined. 

Before we can show that $\cG$ maps $\cB$ into itself, we need some preparatory Lemmas.  

To get bounds the support of $\om$ defined by \eqref{eq:om2}, let us define
\begin{align*}
n_1&:=\min\{x_1\;| (x_1,x_2)\in \supp \om_0\cup \supp\rho_0 \},\\
n_2&:=\max\{x_1\;| (x_1,x_2)\in \supp \om_0\cup \supp\rho_0 \},\\
m&:=\max\{x_2\;| (x_1,x_2)\in \supp \om_0\cup \supp\rho_0 \}.
\end{align*}

\begin{lemma}\label{lem:ombounded}
Let $\zeta>0$ be so small such that $n_1 - \zeta>0$.
There exist a constant $C> 0$ such that for all $\Phi\in \cB$,
$\om$ defined by \eqref{eq:om2} satisfies
$$
\sup_{t\in [0, T]}\|\om(\cdot,t)\|_{L^\infty(\R^2_+)} \leq C.
$$
\end{lemma}

\begin{proof}
Note that $\rho_0$ has a compact support away from the origin. For any $\Phi^{-1}(\by,t) \notin \supp(\rho_0)$, $\rho_0(\Phi^{-1}(\by, t)) = 0$ and hence $|\om(\by, t)|\leq \|\om_0\|_\infty.$
Otherwise $\Phi^{-1}_1(\by, t)\geq n_1$. Note then that $\Phi_1(\Phi^{-1}(\by, t)) \geq \Phi^{-1}(\by, t) - \zeta \geq n_1 -\zeta$ and so 
\begin{align*}
|\om(\by,t)|
&\leq ||\om_0||_\infty + \int_0^t \left|\frac{\rho_0(\Phi^{-1}(\by,t))}{\Phi_1(\Phi^{-1}(\by,t),s)}\right|~ds\leq ||\om_0||_\infty +T\|\rho_0\|_\infty |n_1-\zeta|^{-1}.
\end{align*}
\end{proof}
We can also obtain uniform bounds on $\supp \om(\cdot,t)$ for $t\in[0,T]$ that is presented in the next lemma.
\begin{lemma}\label{SuppwBound}
 For sufficiently small $\zeta>0$, there exists $\fn_1,\fn_2,\fm>0$ such that for all $t\in[0,T]$, we have
\begin{align}\label{suppOmt}
\supp \om(\cdot,t)\subset [\fn_1,\fn_2]\times [0,\fm].
\end{align}
\end{lemma}

\begin{proof} 
Let $\supp(\om_0)\cup \supp(\rho_0) \subset [n_1, n_2]\times [0, m]$, where $n_1 > 0$. First we show that for sufficiently small $\zeta > 0$, the support
of $\om(\cdot, t)$ at each $t\in [0, T]$ is contained in 
$$
[n_1 - \zeta, n_2 + \zeta]\times [0, m+\zeta], 
$$
where we take $\zeta$ small so that $n_1 - \zeta > 0$. 
Observe that from the definition of $\om(\by, t)$, we have that $\om(\by, t) = 0$ if $\Phi^{-1}(\by, t)$ is outside of $[n_1, n_2]\times [0, m]$.
This means that $$\supp(\om(\cdot, t))\subset \Phi([n_1, n_2]\times [0, m],t).$$ 
Now if $\bx = (x_1, x_2)$ and $x_1 \geq n_1$, then 
$$
\Phi_1(\bx, t) = x_1 + \hat \Phi_1(\bx, t) \geq n_1 - \zeta,
$$
since by definition, $\Phi\in \cB$. Similar arguments
lead to
$$
\Phi([n_1, n_2]\times [0, m],t) \subset [n_1 - \zeta, n_2 + \zeta]\times [0, m+\zeta].
$$
The conclusion \eqref{suppOmt} follows by taking any $0 < \fn_1 < n_1-\zeta < n_2+\zeta < \fn_2$ and $\fm > m+\zeta$.
\end{proof}

\begin{lemma}\label{lem:ubounded} There exists a constant $C=C(\zeta,T,\om_0,\rho_0,\alpha)$ such that for all $\Phi\in \cB$, $\bx\in \D$ and $t\in[0,T]$
\begin{align*}
|Q(\bx, t)| \le C.
\end{align*}
As a consequence, $|\bu(\bx,t)|\le C$.for the $\bu$ defined by \eqref{eq:u}.
\end{lemma}
\begin{proof} 
Note that $\bu(\bx,t) =(u_1(\bx,t),u_2(\bx,t))= (-x_1 Q(\bx,t), x_2 Q(\bx,t))$. 
To show that $Q$ is uniformly bounded, we use Lemma \ref{lem:ombounded} and Lemma \ref{SuppwBound}:
\begin{align*}
|Q(\bx,t)|&=\left|\int_{\bx+S_\al} \frac{y_1 y_2}{|\by|^4}\om(\by,t)~d\by\right|
\leq C \int_{\fn_1}^{\fn_2}\int_0^{\fm}  \frac{y_1 y_2}{|\by|^4}~dy_1 dy_2 < C.
\end{align*}
The bound on $\bu$ follows from the bound on $Q$.
\end{proof}

\begin{lemma}\label{lem:nablaubound}
There exists $C=C(\zeta,T,\om_0,\rho_0,\alpha)$ such that for all $\Phi\in \cB$, $\bx\in \D$ and $t\in[0,T]$
\begin{align*}
\mn\nabla \bu (\bx,t)\mn&\le C
\end{align*}
for the $\bu$ defined by \eqref{eq:u}.
Moreover $\bu(\cdot,t)$ and $\nabla \bu(\cdot,t)$ are both Lipschitz, with Lipschitz constants independent of $\Phi \in \cB$.
\end{lemma}
\begin{proof}
To prove boundedness of $\nabla \bu(\bx,t)$,
it remains to show boundedness of the partial derivatives of $Q$. For convenience of notation, introduce $$G(\bx,t;y_1):=\int_{x_2}^{x_2+\alpha(y_1-x_1)}K(\by,t)~dy_2$$ where $K(\bx,t):=x_1 x_2 |\bx|^{-4} \om(\bx,t)$. Note that
\begin{align}\label{eq:rep_Q_G}
Q(\bx, t) = \int_{x_1}^\infty G(\bx, t; y_1)~dy_1.
\end{align}
\noindent Note that by our definition $G(\bx,t;x_1)=0$. Computing
$\d_{1} Q(\bx, t)$ from \eqref{eq:rep_Q_G} and using the bound for $\|\om\|_{\infty}$ from Lemma \ref{lem:ombounded} as well as the bounds $\fn_1, \fn_2, \fm$ for $\supp(\om)$ from Lemma \ref{SuppwBound}, we have
\begin{align*}
|\partial_{x_1} Q(\bx,t)|
        &= \alpha \int_{x_1}^{\infty} \frac{y_1 [x_2+\alpha(y_1-x_1)]}{(y_1^2+[x_2+\alpha(y_1-x_1)]^2)^2}\om(y_1, x_2+\alpha(y_1-x_1), t)~dy_1\\
        &\leq \alpha C \int_{\fn_1}^{\fn_2}\frac{y_1\fm}{y_1^4}~dy_1
        \leq \alpha C \int_{\fn_1}^{\fn_2}\frac{\fn_2 \fm}{\fn_1^4}~dy_1\le C(\fn_1,\fn_2,\fm,\alpha) = C(\zeta,T,\om_0,\rho_0,\alpha).
\end{align*}
In going from the first to the second line of the preceding calculation, we have used that $\om(\cdot, t)$ is zero if $x_2+\alpha(y_1-x_1)$ is larger than $\fm$. Boundedness of $\partial_{x_2}Q$ can be achieved analogously. Therefore, $\nabla\bu$ is bounded. From this, we can easily deduce that $\bu$ is Lipschitz
with uniform Lipschitz constant depending only on $\om_0, \rho_0, \al, \zeta$ and $T$.
Observe that $K$ is uniformly bounded by Lemma \ref{SuppwBound}.
Utilizing the boundedness of $K$, we next show that $G$ is Lipschitz.  
\begin{align*}
&|G(\bx,t;y_1)-G(\bz,t;y_1)|\\
&\leq \left| \int_{x_2}^{x_2+\alpha(y_1-x_1)}K(\by,t)~dy_2 - \int_{z_2}^{x_2+\alpha(y_1-x_1)}K(\by,t)~dy_2\right|+\left|\int_{x_2+\alpha(y_1-x_1)}^{z_2+\alpha(y_1-z_1)}K(\by,t)~dy_2\right| \\
               &\leq C |z_2-x_2|+C(|z_2-x_2|+\alpha|z_1-x_1|)
\leq C|\bx-\bz|.
\end{align*}
Moreover 
\begin{align*}
|Q(\bx,t)-Q(\bz,t)|&\leq \int_{\fn_1}^{\fn_2}|G(\bx,t;y_1)-G(\bz,t;y_1)|~dy_1+\left|\int_{x_1}^{\fn_2}G(\bz,t;y_1)~dy_1-\int_{z_1}^{\fn_2}G(\bz,t;y_1)~dy_1\right|\\
&\leq C|\bx-\bz||\fn_2-\fn_1|+C|x_1-z_1|\leq C|\bx-\bz| 
\end{align*}
where in the second to last line we apply the Lipschitz continuity and boundedness of $G$. On the other hand, since 
$$\partial_{x_1} Q(\bx,t)=G(\bx,t;x_1)-\alpha\int_{x_1}^{\infty}K(y_1,x_2+\alpha(y_1-x_1),t)~dy_1$$
we can see that Lipschitz continuity of $\d_{x_1}Q$ is straightforward from Lipschitz continuity of $G$ and boundedness of $K$. We can analogously establish Lipschitz continuity of $\d_{x_2}Q$ using a similar argument. Combining Lipschitz continuity of both $Q$ and partial derivatives of $Q$, we conclude that each entry in the matrix $\nabla\bu$ is Lipschitz implying that $\nabla\bu$ itself is uniformly Lipschitz.
\end{proof}
\begin{remark} All the constants $C$ in Lemmas \ref{lem:ombounded}--\ref{lem:nablaubound} as well as the Lipschitz constants of $\bu$ and $\nabla \bu$ depend only on $\om_0,\rho_0,\zeta,T$ and $\al$. Moreover the constants remain bounded as $T\to 0$.
\end{remark}

\begin{lemma} If $\Phi\in \cB$ and $\zeta,T>0$ are sufficiently small, then $\cG[\Phi]\in \cB$.
\end{lemma}

\begin{proof}
Let $\Up =\cG[\Phi]$. Clearly, $\Up(\cdot,t)$ is $C^1$ on $\D\times \D$. We need to show that $\Up(\cdot,t)$
maps $\D$ into $\D$ and that the axes are mapped into
themselves. This is a consequence of the bounds 
for $Q$ given in  Lemma \ref{lem:ubounded}.
More precisely, for any $\bx = (x_1, x_2)$ with $x_1, x_2 > 0$ we have
\begin{align*}
\frac{d}{dt}\Up_1(\bx, t) &= - \Up_1(\bx, t) Q(\Up_1(\bx, t), t)\geq - C \Up_1(\bx, t) 
\end{align*}
as long as $\Up_1(\bx, t) > 0$. Integration of
the preceding differential inequality gives
$\Up_1(\bx, t) \geq x_1 e^{- C t}$ for $t\in [0, T]$, so $\Up_1(\bx, t)$ remains positive. The same
argument holds for $\Up_2(\bx, t)$. 

We now need to show that the axes are mapped into themselves. It is clear that $\Up(0, 0, t) = (0, 0)$. For brevity, we only consider the vertical axes.
To show $\Up(x_1, 0, t) \subseteq [0, \infty)\times\{0\}$ for positive $x_1$, we observe first the property $u_2(x_1, 0, t) = 0$. So we only need $\Up(x_1, 0, t) > 0$.
To this end, write again 
\begin{align*}
\frac{d}{dt}\Up_1(x_1, 0, t) &\geq - C
\Up_1(x_1, 0, t)
\end{align*}
and use a similar argument as before. 

It remains to show $\cG[\Phi]=\Id+\hat \Up$ for some $\hat \Up\in C([0,T], C^1(\D,\D))$ with $\|\hat \Up\|\le \zeta$. Obviously, 
$$
\hat \Up(\bx, t) = \int_0^t \bu(\Phi(\bx, s), s)~ds
$$
and this can be bounded by $C T$ and is less than $\zeta$ for sufficiently small $T > 0$.
\end{proof}

\begin{lemma}
$\cG:\cB\to\cB$ is a contraction for sufficiently small $\zeta,T$.
\end{lemma}
\begin{proof}
Let $\Phi,\Psi\in \cB$. Using Lemma \ref{lem:ubounded} we find
\begin{align*}
|\cG[\Phi](\bx,t)-\cG[\Psi](\bx,t)|& = \left|\int_0^t \bu(\Psi(\bx,s),s)- \bu(\Phi(\bx,s),s)~ds\right|\le C\int_0^t |\Psi(\bx,s)-\Phi(\bx,s)|~ds\\
&\le CT\|\hat\Phi-\hat\Psi\|\le CT d(\Phi,\Psi).
\end{align*}
The derivative of $\cG$ is given by
\begin{align*}
\nabla \cG[\Phi](\bx,t)&=I_{2\times 2}+\int_0^t(\nabla \bu)(\Phi(\bx,s),s)(\nabla \Phi)(\bx,s)~ds.
\end{align*}
Inserting cross terms and using Lemma \ref{lem:nablaubound} we get:
\begin{align*}
\mn\nabla \cG[\Phi](\bx,t)-\nabla \cG[\Psi](\bx,t)\mn &\le
\int_0^t \mn (\nabla \bu)(\Psi(\bx,s),s)(\nabla \Psi)(\bx,s)-(\nabla \bu)(\Phi(\bx,s),s)(\nabla \Phi)(\bx,s)\mn\\
&\le\int_0^t \mn (\nabla \bu)(\Psi(\bx,s),s)\mn \mn(\nabla \Psi)(\bx,s)-(\nabla \Phi)(\bx,s)\mn~ds \\
&\qquad+\int_0^t\mn (\nabla \bu)(\Psi(\bx,s),s)-(\nabla \bu)(\Phi(\bx,s),s)\mn \mn (\nabla \Phi)(\bx,s)\mn~ds\\
&\le CT\|\hat\Phi-\hat\Psi\|+CT(1+\zeta) \|\hat\Phi-\hat\Psi\|\le CTd(\Phi,\Psi)
\end{align*}
Together, this implies 
$$d(\cG[\Phi],\cG[\Psi])\le CTd(\Phi,\Psi).$$
Choosing $T$ sufficiently small makes $\cG$ a contraction.
\end{proof}
We can now finish the proof of local existence. 
By the contraction mapping theorem there exists a unique solution $\Phi$ of \eqref{eq:Phi1} on $[0,T]$. Finally, $\Phi$ gives a unique solution of \eqref{model} on a short time interval $[0,T]$ using the relation \eqref{eq:om2}.


\begin{thebibliography}{10}
\bibitem{ChaeConstWu} D. Chae, P. Constantin, J. Wu: An incompressible 2D didactic model with singularity and explicit solutions of the 2D Boussinesq equations.
\newblock {\em  arXiv:1401.7617v2}, 2014.

\bibitem{sixAuthors}
K.~Choi, T.Y.~Hou, A.~Kiselev, G.~Luo, V.~\v{S}ver\'{a}k and Y.~Yao,
\newblock On the finite-time blowup of a 1d model for the 3d axisymmetric
  {Euler} equations.
\newblock {\em arXiv:1407.4776}, 2014.

\bibitem{CKY}
K.~Choi, A.~Kiselev and Y.~Yao,
\newblock Finite time blow up for a 1d model of 2d {Boussinesq} system.
\newblock {\em Comm. Math. Phys.}, 334(3):1667--1679, 2015.

\bibitem{ConstRev}
P.~Constantin:\emph{On the Euler equations of incompressible fluids}, Bull. Amer. Math. Soc, Volume 44, Number 4, October (2007), Pages 603–-621.

\bibitem{ConstantinFefferman} P. Constantin, C. Feﬀerman: \emph{Direction of vorticity and the problem of global regularity for the Navier-Stokes equations.} Indiana Univ. Math. J. 42 (1993), 775. MR1254117 (95j:35169) 

\bibitem{CFM} P. Constantin, C. Feﬀerman, A. Majda \emph{Geometric constraints on potentially singular solutions for the 3-D Euler equations.} Commun. in PDE 21 (1996), 559-571. MR1387460 (97c:35154)

\bibitem{Denisov1}
S. Denisov:
Infinite superlinear growth of the gradient for the two-dimensional Euler equation
, Discrete Contin. Dyn. Syst. A 23 (2009), 755-764.
\bibitem{Denisov2}
 S. Denisov:
Double-exponential growth of the vorticity gradient for the two-dimensional Euler equa-
tion, Proceedings of the AMS, Vol. 143, N3, 2015, 1199-1210.
\bibitem{Denisov3} S. Denisov:
The sharp corner formation in 2D Euler dynamics of patches:  infinite double-exponential
rate of merging, Arch. Rational Mech. Anal., Vol. 215, N2, 2015, 675-705.

\bibitem{Gordon}W. B. Gordon: On the Diffeomorphisms of Euclidean Space,
The American Mathematical Monthly Vol. 79, No. 7, pp. 755-759.

\bibitem{HoangRadosz1D} V. Hoang, M. Radosz:
Blowup with vorticity control for 1D model equations,
in preparation.

\bibitem{HouLuo1}
T.Y.~Hou and G.~Luo,
\newblock Toward the finite-time blowup of the 3d axisymmetric {Euler}
  equations: A numerical investigation,
\newblock {\em Multiscale Model. Simul.}, 12(4):1722--1776, 2014.

\bibitem{HouLuo2}
T.Y.~Hou and G.~Luo,
\newblock Potentially singular solutions of the 3D axisymmetric Euler equations.
\newblock {\em PNAS}, vol. 111 no. 36, 12968-12973,
\emph{DOI 10.1073/pnas.1405238111.}

\bibitem{MajdaBertozzi} A. Majda, A. Bertozzi: Vorticity and Incompressible Flow, Cambridge University Press (2002).

\bibitem{KiselevLenyaYao}
A. Kiselev, L. Ryzhik, Y. Yao and A. Zlato\v{s}.
\newblock Finite time singularity formation for the modified SQG patch equation.
\newblock {\em arXiv:1508.07613}.

\bibitem{kiselev2013small}
A.~Kiselev and V.~\v{S}ver\'{a}k.
\newblock Small scale creation for solutions of the incompressible two
  dimensional {Euler} equation.
\newblock {\em Ann. of Math.(2)}, 180(3):1205--1220, 2014.

\bibitem{KiselevTan} A.~Kiselev and C.~Tan: \emph{Finite time blow up in the hyperbolic boussinesq system}, In preparation.

\bibitem{TT} T.~Tao, \emph{Finite time blowup for Lagrangian modifications of the three-dimensional Euler equation}, preprint \texttt{arXiv:1606.08481v1} (2016)
\end{thebibliography}
\end{document}